\documentclass[12pt, reqno]{amsart}
\usepackage{amssymb}
\usepackage{eucal}
\usepackage{amsmath}
\usepackage{amscd}
\usepackage[dvips]{color}
\usepackage{mathtools}
\usepackage{multicol}
\usepackage[all]{xy}           
\usepackage{graphicx}
\usepackage{color}
\usepackage{colordvi}
\usepackage{xspace}
\usepackage{ulem}
\usepackage{cancel}
\usepackage{tikz}
\usepackage{longtable}
\usepackage{multirow}
\usepackage{ifpdf}
\ifpdf
 \usepackage[colorlinks,final,backref=page,hyperindex]{hyperref}
\else
 \usepackage[colorlinks,final,backref=page,hyperindex,hypertex]{hyperref}
\fi

\begin{document}
\newcommand {\emptycomment}[1]{} 

\newcommand{\tabincell}[2]{\begin{tabular}{@{}#1@{}}#2\end{tabular}}

\newcommand{\nc}{\newcommand}
\newcommand{\delete}[1]{}

\nc{\mlabel}[1]{\label{#1}}  
\nc{\mcite}[1]{\cite{#1}}  
\nc{\mref}[1]{\ref{#1}}  
\nc{\meqref}[1]{~\eqref{#1}} 
\nc{\mbibitem}[1]{\bibitem{#1}} 

\delete{
\nc{\mlabel}[1]{\label{#1}  
{\hfill \hspace{1cm}{\bf{{\ }\hfill(#1)}}}}
\nc{\mcite}[1]{\cite{#1}{{\bf{{\ }(#1)}}}}  
\nc{\mref}[1]{\ref{#1}{{\bf{{\ }(#1)}}}}  
\nc{\meqref}[1]{~\eqref{#1}{{\bf{{\ }(#1)}}}} 
\nc{\mbibitem}[1]{\bibitem[\bf #1]{#1}} 
}

\newtheorem{thm}{Theorem}[section]
\newtheorem{lem}[thm]{Lemma}
\newtheorem{cor}[thm]{Corollary}
\newtheorem{pro}[thm]{Proposition}
\newtheorem{conj}[thm]{Conjecture}
\theoremstyle{definition}
\newtheorem{defi}[thm]{Definition}
\newtheorem{ex}[thm]{Example}
\newtheorem{rmk}[thm]{Remark}
\newtheorem{pdef}[thm]{Proposition-Definition}
\newtheorem{condition}[thm]{Condition}

\renewcommand{\labelenumi}{{\rm(\alph{enumi})}}
\renewcommand{\theenumi}{\alph{enumi}}
\renewcommand{\labelenumii}{{\rm(\roman{enumii})}}
\renewcommand{\theenumii}{\roman{enumii}}

\nc{\tred}[1]{\textcolor{red}{#1}}
\nc{\tblue}[1]{\textcolor{blue}{#1}}
\nc{\tgreen}[1]{\textcolor{green}{#1}}
\nc{\tpurple}[1]{\textcolor{purple}{#1}}
\nc{\btred}[1]{\textcolor{red}{\bf #1}}
\nc{\btblue}[1]{\textcolor{blue}{\bf #1}}
\nc{\btgreen}[1]{\textcolor{green}{\bf #1}}
\nc{\btpurple}[1]{\textcolor{purple}{\bf #1}}


\newcommand{\End}{\text{End}}

\nc{\calb}{\mathcal{B}}
\nc{\call}{\mathcal{L}}
\nc{\calo}{\mathcal{O}}
\nc{\frakg}{\mathfrak{g}}
\nc{\frakh}{\mathfrak{h}}
\nc{\ad}{\mathrm{ad}}

\nc{\ccred}[1]{\tred{\textcircled{#1}}}


\newcommand{\cm}[1]{\textcolor{purple}{\underline{CM:}#1 }}
\newcommand{\gl}[1]{\textcolor{blue}{\underline{GL:}#1 }}

\title[New splittings of operations of Poisson  and transposed Poisson algebras]
{New splittings of operations of Poisson algebras and transposed Poisson algebras and related algebraic structures}

\author{Guilai Liu}
\address{Chern Institute of Mathematics \& LPMC, Nankai University, Tianjin 300071, China}
\email{1120190007@mail.nankai.edu.cn}

\author{Chengming Bai}
\address{Chern Institute of Mathematics \& LPMC, Nankai University, Tianjin 300071, China}
\email{baicm@nankai.edu.cn}


\begin{abstract}
There are two kinds of splittings of operations, namely,
the classical splitting which is  interpreted operadically as
taking successors and another splitting which we call the second
splitting giving the anti-structures of the successors' algebras.
The algebraic structures corresponding to them respectively are
characterized in terms of representations. Due to the appearance
of the two bilinear operations in Poisson algebras and transposed Poisson algebras,
we commence to study new splittings of operations in the ``mixed"
sense that the commutative associative products and Lie brackets
are splitted in different manners respectively, that is, they are
splitted interlacedly in three manners: the classical
splitting, the second splitting and the un-splitting. Accordingly
the corresponding algebraic structures are given. More explicitly,
there are 8 algebraic structures interpreted in terms of
representations of Poisson algebras illustrating the mixed
splittings of operations of Poisson algebras respectively, including the known pre-Poisson algebras. For
illustrating the mixed splittings of operations of transposed
Poisson algebras, there are 8 algebraic structures interpreted in
terms of representations
of transposed Poisson algebras
 on the spaces themselves and another 8
algebraic structures interpreted in terms of representations
of transposed Poisson algebras
on the dual spaces. Moreover, such a phenomenon 
exhibits  an
obvious difference between Poisson algebras and transposed Poisson
algebras.
\end{abstract}

\subjclass[2020]{
    17A36,  
    17A40,  
    17B10, 
    17B40, 
    17B60, 
    17B63,  
    17D25.  
}

\keywords{Poisson algebras; transposed Poisson algebras; splitting}

\maketitle


\tableofcontents

\allowdisplaybreaks

\section{Introduction}
\label{sec:0}\

 This paper aims to introduce and interpret new splittings of
    operations of Poisson algebras and transposed Poisson algebras  in terms of their
    representations, giving various related algebraic
    structures.


\subsection{Classical splitting of operations of Lie algebras and associative algebras}\label{sec:0.1}\

There are many algebraic structures having a property of
``splitting operations", that is, expressing each product of an
algebraic structure as the sum or the (anti)-commutator of the sum
of a string of operations. The   typical examples are pre-Lie
algebras and dendriform algebras which illustrate the spitting of
operations of Lie algebras and associative algebras respectively
``in a coherent way".

\begin{defi}
A  \textbf{pre-Lie algebra} is a pair $(A,\circ)$, such that $A$
is a vector space, and $\circ:A\otimes A\rightarrow A$ is a
bilinear operation satisfying
    \begin{equation}
        (x\circ y)\circ z-x\circ(y\circ z)=(y\circ x)\circ z-y\circ(x\circ z),\;\;\forall x,y,z\in A.
    \end{equation}
\end{defi}

Pre-Lie algebras, also called left-symmetric algebras,  originated
from diverse areas of study, including convex homogeneous cones
 \cite{Vin}, affine manifolds and affine structures on Lie groups
\cite{Ko}, and deformation of associative algebras \cite{Ger}.
They also appear in many fields in mathematics and mathematical
physics, such as symplectic and K\"{a}hler structures on Lie
groups \cite{Chu,Lic1988},  vertex algebras \cite{BK}, quantum
field theory \cite{CK} and operads \cite{CL}, see
\cite{Bai2021,Bur} and the references therein.

Pre-Lie algebras are Lie-admissible algebras, that is, the
commutator of a pre-Lie algebra is a Lie algebra. Hence the
operation of a pre-Lie algebra expresses a kind of splitting the
Lie bracket of a Lie algebra. Moreover, the left multiplication
operators of a pre-Lie algebra give a representation of the
commutator Lie algebra, characterizing the so-called ``coherent
way".

\begin{defi}
A \textbf{dendriform algebra} is a triple $(A,\succ,\prec)$, such
that $A$ is a vector space, and $\succ,\prec:A\otimes A\rightarrow
A$ are bilinear operations satisfying
\begin{equation}\label{dendriform algebra}
    x\succ(y\succ z)=(x\cdot y)\succ z,\ \ (x\prec y)\prec z=x\prec(y\cdot z),\ \ (x\succ y)\prec z=x\succ (y\prec z),
\end{equation}
where $x\cdot y=x\succ y+x\prec y$, for all $x,y,z\in A$. In particular, for a dendriform algebra $(A,\succ,\prec)$, if
\begin{equation}
x\succ y=y\prec x,\;\;\forall x,y\in A,
\end{equation}
then $(A,\star:=\succ)$ is called a {\bf Zinbiel algebra}.
\end{defi}

The notion of dendriform algebras was introduced by Loday in the
study of algebraic K-theory \cite{Lod}, and  they appear in a
lot of fields in mathematics and physics, such as arithmetic
\cite {Lo1}, combinatorics \cite{LR1}, Hopf algebras
\cite{Cha, Ho, Ho1, LR,Ro}, homology \cite{Fra, Fra1}, operads
\cite{Lo2}, Lie and Leibniz algebras \cite{Fra1} and quantum
field theory \cite{Fo}.

The sum of two bilinear operations in a dendriform algebra
$(A,\succ,\prec)$ gives an associative algebra $(A,\cdot)$. Hence
dendriform algebras have a property of splitting the
associativity, that is, expressing the product of an associative
algebra as the sum of two bilinear operations. Such a
decomposition or splitting of the product of an associative
algebra is coherent in the sense that the left and right
multiplication operators of a dendriform algebra give a
representation of the sum associative algebra. Note that in this
sense, pre-Lie algebras and dendriform algebras play similar roles
in the splitting of operations of Lie algebras and associative algebras
respectively.

Furthermore, there is a general theory on the splitting of
operations in the above sense (the so-called coherent way) in
terms of operads in \cite{BBGN}. The notions of successors and
trisuccessors were introduced to interpret the splitting of
operations into the sum of two or three pieces respectively. In
this sense, the operad of pre-Lie algebras is the successor of the
operad of Lie algebras and the operad of dendriform algebras is
the successor of the operad of associative algebras.

To avoid the possible confusion, we refer to this kind of
splitting as the {\bf classical splitting}, that is, the operations
of pre-Lie algebras and dendriform algebras give the classical
splitting of operations of Lie algebras and associative algebras
respectively.

\subsection{Second splitting of operations of Lie algebras and associative algebras}\label{sec:0.2}\

There is another approach of splitting operations introduced as
the ``anti-structures" of the successors' algebras. The first
example is anti-pre-Lie algebras introduced in \cite{LB2022},
giving another splitting of operations of Lie algebras.

\begin{defi} An \textbf{anti-pre-Lie algebra} is a pair $(A,\circ)$, such
that $A$ is a vector space, and $\circ:A\otimes A \rightarrow A$
is a bilinear operation satisfying
    \begin{equation}\label{eq:defi:anti-pre-Lie algebras1}
        x\circ(y\circ z)-y\circ(x\circ z)=[y,x]\circ z,
    \end{equation}
    \begin{equation}\label{eq:defi:anti-pre-Lie algebras2}
        [x,y]\circ z+[y,z]\circ x+[z,x]\circ y=0,
    \end{equation}
    for all $x,y,z\in A$, where the operation $[-,-]:A\otimes A\rightarrow A$ is defined by
    \begin{equation}\label{sub-adj2}
        [x,y]=x\circ y-y\circ x, \;\;\forall x,y\in A.
    \end{equation}
\end{defi}

Anti-pre-Lie algebras are characterized as the Lie-admissible
algebras whose negative left multiplication operators give
representations of their commutator Lie algebras, justifying the
notion due to the comparison with pre-Lie algebras. Hence in this
sense, the operations of anti-pre-Lie algebras give a splitting
of operations of Lie algebras as a kind of ``anti-structures" of
pre-Lie algebras.

Similarly, the notion of anti-dendriform algebras was introduced
in \cite{Gao} as the anti-structures of dendriform algebras, whose
operations give
a splitting of operations of associative algebras which is different from the classical splitting.

\begin{defi}
An \textbf{anti-dendriform algebra}  is a triple $(A,\succ,\prec)$, such that $A$ is a vector space, and $\succ,\prec:A\otimes A\rightarrow A$ are bilinear operations satisfying
\begin{equation}\label{anti-dend 1}
x\succ(y\succ z)=-(x\cdot y)\succ z=-x\prec(y\cdot z)=(x\prec y)\prec z,
\end{equation}
\begin{equation}\label{anti-dend 2}
    (x\succ y)\prec z=x\succ(y\prec z),
\end{equation}
where $x\cdot y=x\succ y+x\prec y$, for all $x,y,z\in A$. In particular, for an anti-dendriform algebra $(A,\succ,\prec)$, if
\begin{equation}
x\succ y=y\prec x,\;\;\forall x,y\in A,
\end{equation}
then $(A,\star:=\succ)$ is called an {\bf anti-Zinbiel algebra}.
\end{defi}

Anti-dendriform algebras keep the property of splitting
associativity, that is, the sum of the two bilinear operations in an anti-dendriform algebra $(A,\succ,\prec)$ gives an associative algebra $(A,\cdot)$.
However it is the negative left and right
multiplication operators of an anti-dendriform algebra that compose a representation of the
sum associative algebra, instead of the left and right
multiplication operators doing so for a dendriform algebra.

To avoid the possible confusion, we refer to this kind of
splitting as the {\bf second splitting}, that is, the operations of
anti-pre-Lie algebras and anti-dendriform algebras give the second splitting
of operations of Lie algebras and associative algebras respectively.

\subsection{New splittings of operations of Poisson algebras and transposed Poisson algebras}\label{sec:0.3}\

Poisson algebras arose in the study of Poisson geometry \cite{BV1,Li77,Wei77}, and are closely related to
a lot of topics in mathematics and physics.
\begin{defi}
    A \textbf{Poisson algebra}  is a triple
    $(A,\cdot,[-,-])$, where $(A,\cdot)$ is a commutative associative
    algebra, $(A,[-,-])$ is a Lie algebra,  and they are compatible in the sense of the Leibniz
    rule:
    \begin{equation}\label{eq:PA}
        [z,x\cdot y]=[z,x]\cdot y+x\cdot[z,y],\;\;\forall x,y,z\in A.
    \end{equation}
\end{defi}

The notion of transposed Poisson algebras was introduced in \cite{Bai2020} as the dual notion of Poisson algebras,
which exchanges the roles of the two bilinear operations in the Leibniz rule defining  Poisson algebras.
They closely relate to a
lot of other algebraic structures such as Novikov-Poisson algebras
\cite{Xu1997} and 3-Lie algebras \cite{Fil} and further
studies are given in \cite{BFK,FKL,LS,YH}.

\begin{defi}\label{defi:transposed Poisson algebra}\
    A \textbf{transposed Poisson algebra} is a triple
    $(A,\cdot,[-,-])$, where $(A,\cdot)$ is a commutative associative
    algebra, $(A,[-,-])$ is a Lie algebra, and they are compatible in the sense of the transposed Leibniz rule:
    \begin{equation}\label{eq:defi:transposed Poisson algebra}
        2z\cdot[x,y]=[z\cdot x,y]+[x,z\cdot y],\;\; \forall x,y,z\in A.
    \end{equation}
\end{defi}

The notion of pre-Poisson algebras was introduced in
\cite{Agu2000.0} to give the classical splitting of operations of
Poisson algebras, that is, they are  the algebraic structures that combine pre-Lie algebras and Zinbiel algebras satisfying certain compatible conditions.

In this paper, we extend this classical splitting of operations of
Poisson algebras to a wide extent, by introducing new splittings
of operations of both Poisson algebras and transposed Poisson algebras.
Note that both Poisson algebras and transposed Poisson algebras have two
bilinear operations, and hence variations of splitting operations
become possible.

In fact, due to the existence of two bilinear operations for Poisson algebras and
transposed Poisson algebras, we consider the new splittings as
``mixed splittings" in the sense that the commutative associative products and Lie brackets
are splitted in different manners respectively. More explicitly,
the commutative associative products and Lie brackets in Poisson algebras and
transposed Poisson algebras are splitted interlacedly in three manners: the classical
splitting, the second splitting and the un-splitting, giving variations of splitting operations.

Since the algebraic structures corresponding to the classical splitting and the second splitting are
characterized in terms of representations, we also characterize the algebraic structures corresponding to the new
splittings of operations of Poisson algebras and transposed Poisson
algebras in terms of representations. Note that a representation
of a Poisson algebra has a natural dual representation. Hence the
    characterization of algebraic structures corresponding to the new splittings of operations of Poisson
    algebras in terms of representations of Poisson algebras on the spaces themselves is the same as that in terms of representations of Poisson algebras on the dual
    spaces. However, the situation is different for transposed Poisson
algebras, that is, one should consider the characterization of the algebraic structures corresponding to the
new splittings of operations of transposed Poisson algebras in terms of representations of transposed Poisson algebras on the
spaces themselves and representations of transposed Poisson algebras on the dual spaces respectively. Such a phenomenon is partly  due to the
fact that there might not exist automatically dual representations
for representations of transposed Poisson algebras (see  Proposition \ref{pro:dual rep TPA}),
exhibiting  an
obvious difference between Poisson algebras and transposed Poisson
algebras.

Therefore there are 8 algebraic structures interpreted in terms of
representations of Poisson algebras illustrating the mixed
splittings of operations of Poisson algebras respectively, including the known pre-Poisson algebras. For
illustrating the mixed splittings of operations of transposed
Poisson algebras, there are 8 algebraic structures interpreted in
terms of representations
of transposed Poisson algebras
on the spaces themselves and another 8
algebraic structures interpreted in terms of representations
of transposed Poisson algebras
on the dual spaces. Note that some of them also correspond to the
 Poisson algebras and transposed Poisson algebras with nondegenerate bilinear forms satisfying certain conditions respectively.

\subsection{Organization of the paper}\label{sec:0.4}\

This paper is organized as follows.

In Section \ref{sec:1}, we recall some facts  on pre-Lie algebras and Zinbiel algebras as well as anti-pre-Lie algebras and anti-Zinbiel algebras,
as the algebraic structures corresponding to the splittings of operations of Lie algebras and commutative associative algebras, which are interpreted in terms of representations of Lie algebras and commutative associative algebras respectively.

In Section \ref{sec:2}, we introduce 8 algebraic structures respectively corresponding to the mixed splittings of operations of Poisson algebras interlacedly in three manners, in terms of representations of Poisson algebras.
Some of them are closely related to the Poisson algebras with nondegenerate bilinear forms satisfying certain conditions.

In Section \ref{sec:3}, we introduce 8 algebraic structures respectively corresponding to the mixed splittings of operations of transposed Poisson algebras interlacedly in three manners, in terms of the representations of transposed Poisson algebras on the spaces themselves.

In Section \ref{sec:4}, we introduce 8 algebraic structures respectively corresponding to the mixed splittings of operations of transposed Poisson algebras interlacedly in three manners, in terms of the representations of transposed Poisson algebras on the dual spaces.
Some of them are closely related to the transposed Poisson algebras with nondegenerate bilinear forms satisfying certain conditions.

Throughout this paper,  unless otherwise specified, all the vector
spaces and algebras are finite-dimensional over a field  of characteristic zero, although many
results and notions remain valid in the infinite-dimensional case.

\section{Splittings of operations of Lie algebras and commutative associative algebras and related algebraic structures}\label{sec:1}\

We recall some facts on pre-Lie algebras and anti-pre-Lie algebras exhibiting the classical splitting and the second splitting of operations of Lie algebras respectively, which are interpreted in terms
of representations of Lie algebras. Similarly, we do so for commutative associative algebras by recalling  some facts on Zinbiel algebras and anti-Zinbiel algebras.

\subsection{Pre-Lie algebras and anti-pre-Lie algebras}\

Recall some basic facts on representations of Lie algebras. A
\textbf{representation of a Lie algebra} $(\mathfrak{g},[-,-])$ is
a pair $(\rho,V)$, such that $V$ is a vector space and
$\rho:\mathfrak{g}\rightarrow\mathfrak{gl}(V)$ is a Lie algebra
homomorphism for the natural Lie algebra structure on
$\mathfrak{gl}(V)=\mathrm{End}(V)$.
In particular, the linear map ${\rm ad}:\mathfrak{g}\rightarrow
\mathfrak{gl}(\mathfrak{g})$  defined by ${\rm ad}(x)(y)=[x,y]$
for all $x,y\in \mathfrak{g}$, gives a representation
$(\mathrm{ad},\mathfrak{g})$, called the \textbf{adjoint
    representation} of $(\mathfrak{g},[-,-])$.

For a vector space $V$ and a linear map $\rho:\mathfrak{g}\rightarrow\mathfrak{gl}(V)$, the pair $(\rho,V)$ is a representation of a Lie algebra $(\mathfrak{g},[-,-])$ if and only if 
$\mathfrak{g}\oplus V$
is a ({\bf semi-direct product}) Lie algebra  by defining the
multiplication
on $\mathfrak{g}\oplus V$ by
\begin{equation}\label{eq:SDLie}
    [(x,u),(y,v)]=([x,y],\rho(x)v-\rho(y)u),\;\;\forall x,y\in \mathfrak{g}, u,v\in V.
\end{equation}
We denote it by $\mathfrak{g}\ltimes_{\rho}V$.

Let $A$ and $V$ be vector spaces. For a linear map $\rho:A\rightarrow\mathrm{End}(V)$, we set $\rho^{*}:A\rightarrow\mathrm{End}(V^{*})$ by
\begin{equation}
    \langle\rho^{*}(x)v^{*},u\rangle=-\langle v^{*},\rho(x)u\rangle,\;\;\forall x\in A, u\in V, v^{*}\in V^{*}.
\end{equation}
Here $\langle\ ,\ \rangle$ is the usual pairing between $V$ and $V^*$.
If $(\rho,V)$ is a representation of a Lie algebra $(\mathfrak{g},[-,-])$, then $(\rho^{*},V^{*})$ is also a representation of $(\mathfrak{g},[-,-])$.
In particular, $(\mathrm{ad}^{*},\mathfrak{g}^{*})$ is a representation of $(\mathfrak{g},[-,-])$.

Recall that a bilinear form $\mathcal{B}$ on a Lie algebra $(\mathfrak{g},[-,-])$ is called \textbf{invariant} if
\begin{equation}\label{Lie invariance}
    \mathcal{B}([x,y],z)=\mathcal{B}(x,[y,z]),\;\;\forall x,y,z\in\mathfrak{g}.
\end{equation}

Suppose that $(\mathfrak{g},[-,-])$ is a Lie algebra. Then the natural nondegenerate symmetric bilinear form $\mathcal{B}_{d}$
on $\mathfrak{g}\oplus\mathfrak{g}^{*}$ defined by
\begin{equation}\label{B_{d}}
    \mathcal{B}_{d}((x,a^{*}),(y,b^{*}))=\langle x,b^{*}\rangle+\langle a^{*},y\rangle,\;\;\forall   x,y\in \mathfrak{g}, a^{*}, b^{*}\in \mathfrak{g}^{*}
\end{equation}
 is invariant on the Lie algebra $\mathfrak{g}\ltimes_{\mathrm{ad}^{*}}\mathfrak{g}^{*}$.

For a vector space $A$ with a bilinear operation $\circ:A\otimes A\rightarrow A$, $(A,\circ)$ is called a \textbf{Lie-admissible algebra} if the bilinear operation $[-,-]:A\otimes A\rightarrow A$ defined by
\begin{equation}\label{sub-adj}
    [x,y]=x\circ y-y\circ x, \;\;\forall x,y\in A
\end{equation}
equips $A$ with a Lie algebra structure. In this case, $(A,[-,-])$ is called the \textbf{sub-adjacent Lie algebra} of $(A,\circ)$. 

For a vector space $A$ together with a bilinear operation $\circ:A\otimes A\rightarrow A$, denote a linear map
$\mathcal{L}_{\circ}:A\rightarrow\mathrm{End}(A)$ by
\begin{equation}
    \mathcal{L}_{\circ}(x)y:=x\circ y,\;\;\forall x,y\in A.
\end{equation}

There is the following characterization of pre-Lie algebras.

\begin{pro} \cite{Bai2021,Bur} \label{pro:pre-Lie}\ Let $A$ be a vector space together with a bilinear operation $\circ:A\otimes A\rightarrow A$. Then the following conditions are equivalent:
    \begin{enumerate}
        \item $(A,\circ)$ is a pre-Lie algebra.
        \item $(A,\circ)$ is a Lie-admissible algebra such that $(\mathcal{L}_{\circ},A)$ is a representation of the sub-adjacent Lie algebra $(A,[-,-])$.
        \item There is a Lie algebra structure on $A\oplus A$ defined by
        \begin{equation}\label{pre-Lie axiom}
            [(x,a),(y,b)]=(x\circ y-y\circ x, x\circ b-y\circ a),\;\;\forall x,y,a,b\in A.
        \end{equation}
    \end{enumerate}
\end{pro}

If a Lie algebra $(\mathfrak g,[-,-])$ is the sub-adjacent Lie algebra of a pre-Lie algebra $(\mathfrak g,\circ)$, then $(\mathfrak g,\circ)$ is called a {\bf compatible pre-Lie algebra} of
$(\mathfrak g,[-,-])$.

Let $\mathcal{B}$ be a nondegenerate skew-symmetric bilinear form on a Lie algebra $(\mathfrak{g},[-,-])$. If $\mathcal{B}$ satisfies
\begin{equation}\label{2-cocycle}
    \mathcal{B}([x,y],z)+\mathcal{B}([y,z],x)+\mathcal{B}([z,x],y)=0,\;\;\forall x,y,z\in\mathfrak{g},
\end{equation}
then we say $\mathcal{B}$ is a \textbf{symplectic form} \cite{Chu,Lic} on $(\mathfrak{g},[-,-])$, and we call the triple $(\mathfrak{g},[-,-],\mathcal{B})$ a \textbf{symplectic Lie algebra}.

\begin{pro} \cite{Chu, Ku1}
    Let $(\mathfrak{g},[-,-],\mathcal{B})$ be a symplectic Lie algebra. Then there exists a compatible pre-Lie algebra $(\mathfrak{g},\circ)$ of $(\mathfrak{g},[-,-])$ defined by
    \begin{equation}\label{bilinear form pre-Lie}
        \mathcal{B}(x\circ y,z)=-\mathcal{B}(y,[x,z]),\;\;\forall x,y,z\in  \mathfrak{g}.
    \end{equation}
Conversely, let $(A,\circ)$ be a pre-Lie algebra and $(A,[-,-])$ be the sub-adjacent Lie algebra. Then the natural nondegenerate skew-symmetric bilinear form $\mathcal{B}_{p}$ defined by
\begin{equation}\label{B_{p}}
    \mathcal{B}_{p}((x,a^{*}),(y,b^{*}))=\langle x,b^{*}\rangle-\langle a^{*},y\rangle,\;\;\forall x,y\in A, a^{*}, b^{*}\in A^{*}
\end{equation}
is a symplectic form on the Lie algebra $A\ltimes_{\mathcal{L}^{*}_{\circ}}A^{*}$.
\end{pro}

Similarly, anti-pre-Lie algebras are also characterized in terms of representations of the sub-adjacent Lie algebras.

\begin{pro} \cite{LB2022} \label{pro:anti-pre-Lie}\
    Let $A$ be a vector space together with a bilinear operation $\circ:A\otimes A\rightarrow A$. Then the following conditions are equivalent:
    \begin{enumerate}
        \item $(A,\circ)$ is an anti-pre-Lie algebra.
        \item $(A,\circ)$ is a Lie-admissible algebra such that $(-\mathcal{L}_{\circ},A)$ is a representation of the sub-adjacent Lie algebra $(A,[-,-])$.
        \item There is a Lie algebra structure on $A\oplus A$ defined by
        \begin{equation}\label{anti-pre-Lie axiom}
            [(x,a),(y,b)]=(x\circ y-y\circ x, y\circ a-x\circ b),\;\;\forall x,y,a,b\in A.
        \end{equation}
    \end{enumerate}
\end{pro}

Similarly, if a Lie algebra $(\mathfrak g,[-,-])$ is the sub-adjacent Lie algebra of an anti-pre-Lie algebra $(\mathfrak g,\circ)$, then $(\mathfrak g,\circ)$ is called a {\bf compatible anti-pre-Lie algebra} of
$(\mathfrak g,[-,-])$.

Recall that a symmetric bilinear form $\mathcal{B}$ on a Lie algebra $(\mathfrak{g},[-,-])$ is called a \textbf{commutative 2-cocycle} \cite{Dzh} if Eq.~(\ref{2-cocycle}) holds, which in the nondegenerate case is the ``symmetric" version of a symplectic form on the Lie algebra $(\mathfrak{g},[-,-])$.

\begin{pro}\cite{LB2022}\label{bf anti-pre-Lie}\
Let $\mathcal{B}$ be a nondegenerate commutative 2-cocycle on a Lie algebra $(\mathfrak{g},[-,-])$. Then there exists a compatible anti-pre-Lie algebra $(\mathfrak{g},\circ)$ of $(\mathfrak{g},[-,-])$ defined by
\begin{equation}\label{bilinear form anti-pre-Lie}
    \mathcal{B}(x\circ y,z)=\mathcal{B}(y,[x,z]),\;\;\forall x,y,z\in \mathfrak{g}.
\end{equation}
Conversely, let $(A,\circ)$ be an anti-pre-Lie algebra and $(A,[-,-])$ be the sub-adjacent Lie algebra. Then the natural nondegenerate  symmetric bilinear form $\mathcal{B}_{d}$ defined by Eq.~(\ref{B_{d}})
is a commutative 2-cocycle on the Lie algebra $A\ltimes_{-\mathcal{L}^{*}_{\circ}}A^{*}$.
    \end{pro}

\subsection{Zinbiel algebras and anti-Zinbiel algebras}\

Recall some basic facts on representations of commutative associative algebras.
A \textbf{representation of a commutative associative
algebra} $(A,\cdot)$ is a pair $(\mu,V)$, where $V$ is a vector
space and $\mu:A\rightarrow\mathrm{End}(V)$ is a linear map
satisfying
\begin{equation}
    \mu(x\cdot y)=\mu(x)\mu(y),\;\;\forall x,y\in A.
\end{equation}
For a commutative associative algebra $(A,\cdot)$, $(\mathcal{L}_{\cdot},A)$ is a representation of $(A,\cdot)$,
called the \textbf{adjoint representation} of $(A,\cdot)$.

In fact, $(\mu,V)$ is a representation of a commutative
associative algebra $(A,\cdot)$  if and only if the direct sum
$A\oplus V$ of vector spaces is a ({\bf
    semi-direct product}) commutative associative algebra  by defining the multiplication on $A\oplus
V$ by
\begin{equation}\label{eq:SDASSO}
    (x,u)\cdot(y,v)=(x\cdot y,\mu(x)v+\mu(y)u),\;\;\forall x,y\in A, u,v\in V.
\end{equation}
We denote it by $A\ltimes_{\mu}V$.

If $(\mu,V)$ is a representation of a commutative associative
algebra $(A,\cdot)$, then $(-\mu^{*},V^{*})$ is also a
representation of $(A,\cdot)$. In particular,
$(-\mathcal{L}^{*}_{\cdot},A^{*})$ is a representation of
$(A,\cdot)$.

Recall that a bilinear form $\mathcal{B}$ on a (commutative) associative algebra $(A,\cdot)$ is called \textbf{invariant} if
\begin{equation}\label{comm asso invariance}
    \mathcal{B}(x\cdot y,z)=\mathcal{B}(x,y\cdot z),\;\;\forall x,y,z\in A.
\end{equation}

    Let $(A,\cdot)$ be a commutative associative algebra.
    Then the natural nondegenerate symmetric bilinear form $\mathcal{B}_{d}$ defined by Eq.~(\ref{B_{d}}) is invariant on the commutative associative algebra $A\ltimes_{-\mathcal{L}^{*}_{\cdot}}A^{*}$.

For a vector space $A$ together with a bilinear operation $\star:A\otimes A\rightarrow A$, if the bilinear operation  $\cdot:A\otimes A\rightarrow A$ defined by
\begin{equation}\label{comm op}
    x\cdot y=x\star y+y\star x,\;\;\forall x,y\in A
\end{equation}
equips $A$ with a commutative associative algebra structure, then we say $(A,\cdot)$ is the \textbf{sub-adjacent commutative associative algebra} of $(A,\star)$.

Zinbiel algebras and anti-Zinbiel algebras play a similar role for commutative associative algebras as pre-Lie algebras and anti-pre-Lie algebras do for Lie algebras respectively. The notion of Zinbiel algebras is rewritten
in a more straightforward manner as follows.

\begin{defi} \cite{Lod}
    Let $A$ be a vector space together with a bilinear operation $\star:A\otimes A\rightarrow A$. $(A,\star)$ is called a \textbf{Zinbiel algebra} if
    \begin{equation}
        x\star(y\star z)=(x\star y)\star z+(y\star x)\star z,\;\;\forall x,y,z\in A.
    \end{equation}
\end{defi}


\begin{pro} \cite{Bai2010} \label{pro:Zinbiel}\
    Let $A$ be a vector space together with a bilinear operation $\star:A\otimes A\rightarrow A$. Then the following conditions are equivalent:
    \begin{enumerate}
        \item $(A,\star)$ is a Zinbiel algebra.
        \item $(A,\cdot)$ with the bilinear operation $\cdot$ defined by Eq.~(\ref{comm op}) is a commutative associative algebra and
        $(\mathcal{L}_{\star},A)$ is a representation of $(A,\cdot)$.
        \item There is a commutative associative algebra structure on $A\oplus A$ defined by
        \begin{equation}\label{Zinbiel axiom}
        (x,a)\cdot(y,b)=(x\star y+y\star x,x\star b+y\star a),\;\;\forall x,y,a,b\in A.
        \end{equation}
    \end{enumerate}
\end{pro}

If a commutative associative  algebra $(A,\cdot)$ is the sub-adjacent commutative associative algebra of a Zinbiel algebra $(A,\star)$, then $(A,\star)$ is  called a {\bf compatible Zinbiel algebra} of
$(A,\cdot)$.

Recall that a   skew-symmetric bilinear form $\mathcal{B}$ on a (commutative) associative algebra is called a \textbf{Connes cocycle} \cite{Bai2010} if
\begin{equation}\label{Connes cocycle}
    \mathcal{B}(x\cdot y,z)+\mathcal{B}(y\cdot z,x)+\mathcal{B}(z\cdot x,y)=0,\;\;\forall x,y,z\in A.
\end{equation}

\begin{pro} \cite{Bai2010}
Let $\mathcal{B}$ be a nondegenerate Connes cocycle on a commutative associative algebra $(A,\cdot)$. Then there is a compatible Zinbiel algebra $(A,\star)$ of $(A,\cdot)$ defined by
\begin{equation}\label{bilinear form Zinbiel}
    \mathcal{B}(x\star y,z)=\mathcal{B}(y,x\cdot z),\;\;\forall x,y,z\in A.
\end{equation}
Conversely, let $(A,\star)$ be a Zinbiel algebra and $(A,\cdot)$ be the sub-adjacent commutative associative algebra. Then the natural nondegenerate  skew-symmetric bilinear form $\mathcal{B}_{p}$ defined by Eq.~(\ref{B_{p}})
is a Connes cocycle on the commutative associative algebra $A\ltimes_{-\mathcal{L}^{*}_{\star}}A^{*}$.
\end{pro}

Similarly, the notion of anti-Zinbiel algebras is rewritten
in a more straightforward manner as follows.

\begin{defi}
    Let $A$ be a vector space together with a bilinear operation $\star:A\otimes A\rightarrow A$. $(A,\star)$ is called an \textbf{anti-Zinbiel algebra} if
    \begin{equation}
        x\star(y\star z)=-(x\star y+y\star x)\star z=x\star(z\star y),\;\;\forall x,y,z\in A.
    \end{equation}
\end{defi}



\begin{pro}\label{pro:anti-Zinbiel} \cite{Gao} \
    Let $A$ be a vector space together with a bilinear operation $\star:A\otimes A\rightarrow A$. Then the following conditions are equivalent:
    \begin{enumerate}
        \item $(A,\star)$ is an anti-Zinbiel algebra.
        \item $(A,\cdot)$ with the bilinear operation $\cdot$ defined by Eq.~(\ref{comm op}) is a commutative associative algebra and
        $(-\mathcal{L}_{\star},A)$ is a representation of $(A,\cdot)$.
        \item There is a commutative associative algebra structure on $A\oplus A$ defined by
        \begin{equation}\label{anti-Zinbiel axiom}
            (x,a)\cdot(y,b)=(x\star y+y\star x,-x\star b-y\star a),\;\;\forall x,y,a,b\in A.
        \end{equation}
    \end{enumerate}
\end{pro}

Similarly, if a commutative associative  algebra $(A,\cdot)$ is the sub-adjacent commutative associative algebra of an anti-Zinbiel algebra $(A,\star)$, then $(A,\star)$ is  called a {\bf compatible anti-Zinbiel algebra} of
$(A,\cdot)$.

A   symmetric bilinear form $\mathcal{B}$ on a (commutative) associative algebra $(A,\cdot)$ is called a \textbf{commutative Connes cocycle} \cite{Gao} if Eq.~(\ref{Connes cocycle}) holds.

\begin{pro}\cite{Gao}
    Let $\mathcal{B}$ be a nondegenerate commutative Connes cocycle on a commutative associative algebra $(A,\cdot)$. Then there is a compatible anti-Zinbiel algebra $(A,\star)$ of $(A,\cdot)$ defined by
    \begin{equation}\label{bilinear form anti-Zinbiel}
        \mathcal{B}(x\star y,z)=-\mathcal{B}(y,x\cdot z),\;\;\forall x,y,z\in A.
    \end{equation}
Conversely, let $(A,\star)$ be an anti-Zinbiel algebra and $(A,\cdot)$ be the sub-adjacent commutative associative algebra. Then the natural nondegenerate   symmetric bilinear form $\mathcal{B}_{d}$ defined by Eq.~(\ref{B_{d}})
is a commutative Connes cocycle on the commutative associative algebra $A\ltimes_{\mathcal{L}^{*}_{\star}}A^{*}$.
\end{pro}

\section{Mixed splittings of operations of Poisson algebras and related algebraic structures}\label{sec:2}\

 At first we recall some facts on representations of Poisson algebras. Then we introduce 8 algebraic structures respectively corresponding to the mixed splitting of  the commutative associative
 products and Lie brackets of Poisson algebras interlacedly in three manners: the classical splitting, the second splitting and the un-splitting, in terms of representations of Poisson algebras. Finally the relationships between
Poisson algebras with nondegenerate bilinear forms satisfying certain conditions and  some algebraic structures are given.

    \begin{defi}
        A \textbf{representation of a Poisson algebra} $(A,\cdot,[-,-])$ is a triple $(\mu,\rho,V)$, such that $(\mu,V)$ is a representation of the commutative associative algebra $(A,\cdot)$, $(\rho,V)$ is a representation of the Lie algebra $(A,[-,-])$, and the following compatible conditions hold:
        \begin{equation}\label{2}
            \rho(x\cdot y)=\mu(x)\rho(y)+\mu(y)\rho(x),
        \end{equation}
        \begin{equation}\label{3}
            \mu([x,y])=\rho(x)\mu(y)-\mu(y)\rho(x),
        \end{equation}
     for all $x,y\in A$.
    \end{defi}

Let $(A,\cdot,[-,-])$ be a Poisson algebra. Then $(\mathcal{L}_{\cdot},\mathrm{ad},A)$ is a representation of $(A,\cdot,[-,-])$, called the \textbf{adjoint representation} of $(A,\cdot,[-,-])$.
Moreover, $(\mu,\rho,V)$ is a representation of a   Poisson algebra $(A,\cdot,[-,-])$ if and only if the direct sum $A\oplus V$ of vector spaces is a (\textbf{semi-direct product}) Poisson algebra by defining the multiplications on $A\oplus V$ by Eqs.~(\ref{eq:SDASSO}) and (\ref{eq:SDLie})   respectively. We denote it by $A\ltimes_{\mu,\rho}V$.

\begin{pro} \cite{NB} \label{pro:Poisson rep}\
        Let $(A,\cdot,[-,-])$ be a Poisson algebra. If $(\mu,\rho,V)$ is a representation of $(A,\cdot,[-,-])$, then $(-\mu^{*},\rho^{*},V^{*})$ is also a representation of $(A,\cdot,[-,-])$.
\end{pro}

Hence we get the following conclusion.

\begin{cor}
Let $(A,\cdot,[-,-])$ be a Poisson algebra. Then $(-\mathcal{L}^{*}_{\cdot},\mathrm{ad}^{*},A^{*})$ is a representation of $(A,\cdot,[-,-])$, and the natural nondegenerate symmetric bilinear form $\mathcal{B}_{d}$ defined by Eq.~(\ref{B_{d}}) on the resulting Poisson algebra $A\ltimes_{-\mathcal{L}^{*}_{\cdot},\mathrm{ad}^{*}}A^{*}$ is invariant on both the commutative associative algebra $A\ltimes_{-\mathcal{L}^{*}_{\cdot}}A^{*}$ and the Lie algebra $A\ltimes_{\mathrm{ad}^{*}}A^{*}$.
    \end{cor}

Next we introduce 8 algebraic structures corresponding to the mixed splitting of the commutative associative
 products and Lie brackets of Poisson algebras interlacedly in three manners: the classical splitting, the second splitting and the un-splitting, in terms of representations of Poisson algebras.
Note that due to Proposition \ref{pro:Poisson rep}, the characterization of these algebraic structures in terms of representations of Poisson algebras on the dual spaces is the same as that on the spaces themselves.
Before we introduce these various algebraic structures, we give the following ``principle" to name them.
\begin{enumerate}
\item Every algebraic structure here is named by 3 capital letters.
\item The first letter is unified to be ``P" since these algebras are related to Poisson algebras.
\item The second letter denotes the operation corresponding to the splitting of the commutative associative products. Explicitly,  the capital letters ``C", ``Z" and ``A" respectively denote the operations of
 commutative associative algebras, Zinbiel algebras and anti-Zinbiel algebras, corresponding to the un-splitting, the classical splitting and the second splitting.
\item The third letter denotes the operation corresponding to the splitting of the Lie brackets. Explicitly,  the capital letters ``L", ``P" and ``A" respectively denote the operations of
 Lie algebras, pre-Lie algebras and anti-pre-Lie algebras, corresponding to the un-splitting, the classical splitting and the second splitting.
\end{enumerate}

Note that the PZP algebras combining Zinbiel algebras and pre-Lie algebras are exactly the pre-Poisson algebras introduced in \cite{Agu2000.0}.


\subsection{PCP algebras}

\begin{defi}
    A \textbf{PCP algebra} is a triple $(A,\cdot,\circ)$, such that $(A,\cdot)$ is a commutative associative algebra, $(A,\circ)$ is a pre-Lie algebra, and the following equations hold:
    \begin{equation}\label{PCP1}
        (x\cdot y)\circ z=x\cdot(y\circ z)+y\cdot(x\circ z),
    \end{equation}
    \begin{equation}\label{PCP2}
        (x\circ y-y\circ x)\cdot z=x\circ(y\cdot z)-y\cdot(x\circ z),
    \end{equation}
    \begin{equation}\label{PCP3}
        z\circ(x\cdot y)-z\cdot(x\circ y)-z\cdot(y\circ x)=0,
    \end{equation}
 for all $x,y,z\in A$.
\end{defi}

\begin{pro}
    Let $(A,\cdot,[-,-])$ be a Poisson algebra and $(A,\circ)$ be a compatible pre-Lie algebra of $(A,[-,-])$. If $(\mathcal{L}_{\cdot},\mathcal{L}_{\circ},A)$ is a representation of $(A,\cdot,[-,-])$, then $(A,\cdot,\circ)$ is a PCP algebra. Conversely, let $(A,\cdot,\circ)$ be a PCP algebra and $(A,[-,-])$ be the sub-adjacent Lie algebra of $(A,\circ)$. Then $(A,\cdot,[-,-])$ is a Poisson algebra with a representation $(\mathcal{L}_{\cdot},\mathcal{L}_{\circ},A)$. In this case, we say $(A,\cdot,[-,-])$ is the \textbf{sub-adjacent Poisson algebra} of $(A,\cdot,\circ)$, and $(A,\cdot,\circ)$ is a \textbf{compatible  PCP algebra} of $(A,\cdot,[-,-])$.
    \end{pro}
\begin{proof}
   Since $(\mathcal{L}_{\cdot},\mathcal{L}_{\circ},A)$ is a representation of $(A,\cdot,[-,-])$, we get Eqs.~(\ref{PCP1})-(\ref{PCP2}). Moreover,
    by Eq.~(\ref{PCP2}), we have
    \begin{equation}\label{PCP4}
        x\circ(y\cdot z)-y\cdot(x\circ z)=-y\circ(x\cdot z)+x\cdot(y\circ z),\;\;\forall x,y,z\in A.
    \end{equation}
    Thus for all $x,y,z\in A$, we have
    \begin{eqnarray*}
        0&&=[z,x\cdot y]+[x,z]\cdot y+[y,z]\cdot x\\
        &&\overset{(\ref{PCP2})}{=}z\circ(x\cdot y)-(x\cdot y)\circ z+x\circ(z\cdot y)-z\cdot(x\circ y)+y\circ(z\cdot x)-z\cdot(y\circ x)\\
        &&\overset{(\ref{PCP4})}{=}z\circ(x\cdot y)-(x\cdot y)\circ z+x\cdot(y\circ z)+y\cdot(x\circ z)-z\cdot(x\circ y)-z\cdot(y\circ x)\\
        &&\overset{(\ref{PCP1})}{=}z\circ(x\cdot y)-z\cdot(x\circ y)-z\cdot(y\circ x).
    \end{eqnarray*}
Hence Eq.~(\ref{PCP3}) holds. Thus $(A,\cdot,\circ)$ is a PCP algebra. The converse part is proved similarly.
\end{proof}

Hence we get the following conclusion.

\begin{cor}
Let $A$ be a vector space with two bilinear operations $\cdot,\circ:A\otimes A\rightarrow A$. Then the following conditions are equivalent:
\begin{enumerate}
    \item\label{pro:PCP1} $(A,\cdot,\circ)$ is a   PCP algebra.
    \item\label{pro:PCP2} The triple $(A,\cdot,[-,-])$ is a Poisson algebra with a representation $(\mathcal{L}_{\cdot},\mathcal{L}_{\circ},A)$, where $[-,-]$ is defined by Eq.~(\ref{sub-adj}).
    \item\label{pro:PCP3} There is a Poisson algebra structure on $A\oplus A$ in which  the commutative associative product $\cdot$ is defined by
    \begin{equation}\label{comm asso axiom}
    (x,a)\cdot(y,b)=(x\cdot y, x\cdot b+a\cdot y),\;\;\forall x,y,a,b\in A,
    \end{equation}
and the Lie bracket $[-,-]$ is defined by Eq.~(\ref{pre-Lie axiom}).
\end{enumerate}
\end{cor}

\subsection{PCA algebras}

\begin{defi}
    A \textbf{PCA algebra} is a triple $(A,\cdot,\circ)$, such that $(A,\cdot)$ is a commutative associative algebra, $(A,\circ)$ is an anti-pre-Lie algebra, and
    Eq.~(\ref{PCP1}) and the following equations hold:
    \begin{equation}\label{PCA2}
        (x\circ y-y\circ x)\cdot z=y\cdot(x\circ z)-x\circ(y\cdot z),
    \end{equation}
\begin{equation}\label{PCA3}
    z\circ(x\cdot y)+z\cdot (x\circ y)+z\cdot(y\circ x)-2(x\cdot y)\circ z=0,
\end{equation}
for all $x,y,z\in A$.
\end{defi}

\begin{pro}
    Let $(A,\cdot,[-,-])$ be a Poisson algebra and $(A,\circ)$ be a compatible anti-pre-Lie algebra of $(A,[-,-])$. If $(\mathcal{L}_{\cdot},-\mathcal{L}_{\circ},A)$ is a representation of $(A,\cdot,[-,-])$, then $(A,\cdot,\circ)$ is a PCA algebra. Conversely, let $(A,\cdot,\circ)$ be a PCA algebra and $(A,[-,-])$ be the sub-adjacent Lie algebra of $(A,\circ)$. Then $(A,\cdot,[-,-])$ is a Poisson algebra with a representation $(\mathcal{L}_{\cdot},-\mathcal{L}_{\circ},A)$. In this case, we say $(A,\cdot,[-,-])$ is the \textbf{sub-adjacent Poisson algebra} of $(A,\cdot,\circ)$, and $(A,\cdot,\circ)$ is a \textbf{compatible  PCA algebra} of $(A,\cdot,[-,-])$.
\end{pro}
\begin{proof}
  Since $(\mathcal{L}_{\cdot},-\mathcal{L}_{\circ},A)$ is a representation of $(A,\cdot,[-,-])$, we get Eqs.~(\ref{PCP1}) and (\ref{PCA2}). By Eq.~(\ref{PCA2}), Eq.~(\ref{PCP4}) holds. Thus for all $x,y,z\in A$, we have
    \begin{eqnarray*}
        0&&=[z,x\cdot y]+[x,z]\cdot y+[y,z]\cdot x\\
        &&\overset{(\ref{PCA2})}{=}z\circ(x\cdot y)-(x\cdot y)\circ z-x\circ(z\cdot y)+z\cdot(x\circ y)-y\circ(z\cdot x)+z\cdot(y\circ x)\\
        &&\overset{(\ref{PCP4})}{=}z\circ(x\cdot y)-(x\cdot y)\circ z-x\cdot(y\circ z)-y\cdot(x\circ z)+z\cdot(x\circ y)+z\cdot(y\circ x)\\
        &&\overset{(\ref{PCP1})}{=}z\circ(x\cdot y)+z\cdot (x\circ y)+z\cdot(y\circ x)-2(x\cdot y)\circ z.
    \end{eqnarray*}
Hence Eq.~(\ref{PCA3}) holds. So $(A,\cdot,\circ)$ is a PCA algebra. The converse part is proved similarly.
\end{proof}

Hence we get the following conclusion.

\begin{cor}
Let $A$ be a vector space with two bilinear operations $\cdot,\circ:A\otimes A\rightarrow A$. Then the following conditions are equivalent:
    \begin{enumerate}
        \item\label{pro:PCA1} $(A,\cdot,\circ)$ is a  PCA algebra.
        \item\label{pro:PCA2} The triple $(A,\cdot,[-,-])$ is a Poisson algebra with a representation $(\mathcal{L}_{\cdot},-\mathcal{L}_{\circ},A)$, where $[-,-]$ is defined by Eq.~(\ref{sub-adj}).
        \item\label{pro:PCA3} There is a Poisson algebra structure on $A\oplus A$ in which the commutative associative product $\cdot$ is defined by Eq.~(\ref{comm asso axiom})
        and the Lie bracket $[-,-]$ is defined by Eq.~(\ref{anti-pre-Lie axiom}).
    \end{enumerate}
\end{cor}

\begin{pro}\label{bf PCP}
    Let $(A,\cdot,[-,-])$ be a Poisson algebra. Suppose that $\mathcal{B}$ is a nondegenerate symmetric bilinear form on $A$ such that it is  invariant on $(A,\cdot)$ and a commutative 2-cocycle on $(A,[-,-])$. Then there is a compatible PCA algebra $(A,\cdot,\circ)$ in which $\circ$ is defined by Eq.~(\ref{bilinear form anti-pre-Lie}).
    Conversely, let $(A,\cdot,\circ)$ be a  PCA algebra and the sub-adjacent Poisson algebra be $(A,\cdot,[-,-])$. Then there is a Poisson algebra $A\ltimes_{-\mathcal{L}^{*}_{\cdot},-\mathcal{L}^{*}_{\circ}}A^{*}$, and the natural nondegenerate symmetric bilinear form $\mathcal{B}_{d}$ defined by Eq.~(\ref{B_{d}})
is invariant on the commutative associative algebra $A\ltimes_{-\mathcal{L}^{*}_{\cdot}}A^{*}$ and a commutative 2-cocycle on the Lie algebra $A\ltimes_{-\mathcal{L}^{*}_{\circ}}A^{*}$.
    \end{pro}
\begin{proof}
Since $(A,\cdot,[-,-])$ is a Poisson algebra, the following equation holds:
\begin{equation}\label{Poisson equality}
    [x,y\cdot z]+[y,z\cdot x]+[z,x\cdot y]=0,\;\;\forall x,y,z\in A.
\end{equation}
Let $\mathcal{B}$ be a nondegenerate symmetric bilinear form on $A$ such that it is  invariant on $(A,\cdot)$ and a commutative 2-cocycle on $(A,[-,-])$. Then
\begin{eqnarray*}
&&\mathcal{B}((x\cdot y)\circ z-x\cdot(y\circ z)-y\cdot(x\circ z), w)\\
&&\overset{(\ref{bilinear form anti-pre-Lie}),(\ref{comm asso invariance})}{=}\mathcal{B}(z,[x\cdot y,w]-[y,x\cdot w]-[x,y\cdot w])\overset{(\ref{Poisson equality})}{=}0,\\
&&\mathcal{B}((x\circ y-y\circ x)\cdot z-y\cdot(x\circ z)-x\circ(y\cdot z),w)\\
&&\overset{(\ref{bilinear form anti-pre-Lie}),(\ref{comm asso invariance})}{=}\mathcal{B}(z,[x,y]\cdot w-[x,y\cdot w]-y\cdot[x,w])\overset{(\ref{eq:PA})}{=}0.
\end{eqnarray*}
Hence Eqs.~(\ref{PCP1}) and (\ref{PCA2}) hold by the nondegeneracy of $\mathcal{B}$. Thus Eq.~(\ref{PCA3}) holds such that $(A,\cdot,\circ)$ is a PCA algebra.   Conversely, let $(A,\cdot,\circ)$ be a   PCA algebra  and the sub-adjacent Poisson algebra be $(A,\cdot,[-,-])$. Then $(\mathcal{L}_{\cdot},-\mathcal{L}_{\circ},A)$ is a representation of $(A,\cdot,[-,-])$. By Proposition \ref{pro:Poisson rep}, $(-\mathcal{L}^{*}_{\cdot},-\mathcal{L}^{*}_{\circ},A^{*})$ is also a representation of $(A,\cdot,[-,-])$. Thus there is a Poisson algebra structure $A\ltimes_{-\mathcal{L}^{*}_{\cdot},-\mathcal{L}^{*}_{\circ}}A^{*}$.
It is straightforward to show that $\mathcal{B}_{d}$
is  invariant on the commutative associative algebra $A\ltimes_{-\mathcal{L}^{*}_{\cdot}}A^{*}$ and a commutative 2-cocycle on the Lie algebra $A\ltimes_{-\mathcal{L}^{*}_{\circ}}A^{*}$.
    \end{proof}

\subsection{PZL algebras}

\begin{defi}
    A \textbf{PZL algebra} is a triple $(A,\star,[-,-])$, such that $(A,\star)$ is a Zinbiel algebra, $(A,[-,-])$ is a Lie algebra, and the following equations hold:
    \begin{equation}\label{ZLP1}
        [x\star y+y\star x,z]=x\star[y,z]+y\star[x,z],
    \end{equation}
    \begin{equation}\label{ZLP2}
        [x,y]\star z=[x,y\star z]-y\star[x,z],
    \end{equation}
    for all $x,y,z\in A$.
\end{defi}

\begin{pro}\label{pro:PZL}
    Let $(A,\cdot,[-,-])$ be a Poisson algebra and $(A,\star)$ be a compatible Zinbiel algebra of $(A,\cdot)$. If $(\mathcal{L}_{\star},\mathrm{ad},A)$ is a representation of $(A,\cdot,[-,-])$, then $(A,\star,[-,-])$ is a PZL algebra.
    Conversely, let $(A,\star,[-,-])$ be a PZL algebra and $(A,\cdot)$ be the sub-adjacent commutative associative algebra of $(A,\star)$. Then $(A,\cdot,[-,-])$ is a Poisson algebra with a representation $(\mathcal{L}_{\star},\mathrm{ad},A)$. In this case, we say $(A,\cdot,[-,-])$ is the \textbf{sub-adjacent Poisson algebra} of $(A,\star,[-,-])$, and $(A,\star,[-,-])$ is a \textbf{compatible PZL algebra} of $(A,\cdot,[-,-])$.
\end{pro}
\begin{proof}
    We only prove the latter. For all $x,y,z\in A$, we have
    \begin{eqnarray*}
        &&[z,x\cdot y]+[x,z]\cdot y+[y,z]\cdot x\\
        &&=[z,x\star y]+[z,y\star x]+[x,z]\star y+y\star[x,z]+[y,z]\star x+x\star[y,z]\overset{(\ref{ZLP2})}{=}0.
    \end{eqnarray*}
    Thus $(A,\cdot,[-,-])$ is a Poisson algebra. Moreover, by Eqs.~(\ref{ZLP1})-(\ref{ZLP2}), $(\mathcal{L}_{\star},\mathrm{ad},A)$ is a representation of $(A,\cdot,[-,-])$.
\end{proof}

Hence we get the following conclusion.

\begin{cor}
    Let $A$ be a vector space with two bilinear operations $\star,[-,-]:A\otimes A\rightarrow A$. Then the following conditions are equivalent:
    \begin{enumerate}
        \item $(A,\star,[-,-])$ is a PZL algebra.
        \item The triple $(A,\cdot,[-,-])$  is a Poisson algebra with a representation $(\mathcal{L}_{\star},\mathrm{ad},A)$, where $\cdot$ is defined by Eq.~(\ref{comm op}).
        \item There is a Poisson algebra structure on $A\oplus A$ in which the commutative associative product $\cdot$ is
        defined by Eq.~(\ref{Zinbiel axiom}) and the Lie bracket $[-,-]$ is defined by
        \begin{equation}\label{Lie axiom}
            [(x,a),(y,b)]=([x,y],[x,b]-[y,a]),\;\;\forall x,y,a,b\in A.
        \end{equation}
    \end{enumerate}
    \end{cor}

\subsection{PZP algebras or pre-Poisson algebras}


\begin{defi}\cite{Agu2000.0}
    A \textbf{pre-Poisson algebra} or a {\bf PZP algebra} is a triple $(A,\star,\circ)$, such that $(A,\star)$ is a Zinbiel algebra, $(A,\circ)$ is a pre-Lie algebra, and the following equations hold:
    \begin{equation}\label{pre-Poisson 1}
        (x\star y+y\star x)\circ z=x\star(y\circ z)+y\star(x\circ z),
    \end{equation}
    \begin{equation}\label{pre-Poisson 2}
        (x\circ y-y\circ x)\star z=x\circ(y\star z)-y\star(x\circ z),
    \end{equation}
    for all $x,y,z\in A$.
\end{defi}

\begin{pro}\cite{Agu2000.0}
    Let $(A,\cdot,[-,-])$ be a Poisson algebra, $(A,\star)$ be a compatible Zinbiel algebra of $(A,\cdot)$ and $(A,\circ)$ be a compatible pre-Lie algebra of $(A,[-,-])$. If $(\mathcal{L}_{\star},\mathcal{L}_{\circ}, A)$ is a representation of $(A,\cdot,[-,-])$, then $(A,\star,\circ)$ is a pre-Poisson algebra.
    Conversely, let $(A,\star,\circ)$ be a pre-Poisson algebra, $(A,\cdot)$ be the sub-adjacent commutative associative algebra of $(A,\star)$ and $(A,[-,-])$ be the sub-adjacent Lie algebra of $(A,\circ)$. Then $(A,\cdot,[-,-])$ is a Poisson algebra with a representation $(\mathcal{L}_{\star},\mathcal{L}_{\circ}, A)$. In this case, we say $(A,\cdot,[-,-])$ is the \textbf{sub-adjacent Poisson algebra} of $(A,\cdot,\circ)$, and $(A,\cdot,\circ)$ is a \textbf{compatible pre-Poisson algebra} of $(A,\cdot,[-,-])$.
\end{pro}

Hence we get the following conclusion.

\begin{cor}
    Let $A$ be a vector space with two bilinear operations $\star,\circ:A\otimes A\rightarrow A$. Then the following conditions are equivalent:
    \begin{enumerate}
        \item $(A,\star,\circ)$ is a pre-Poisson algebra.
        \item The triple $(A,\cdot,[-,-])$is a Poisson algebra with a representation $(\mathcal{L}_{\star},\mathcal{L}_{\circ},A)$, where $\cdot$ and $[-,-]$ are respectively defined by Eqs.~(\ref{comm op}) and (\ref{sub-adj}).
        \item There is a Poisson algebra structure on $A\oplus A$ in which the commutative associative product $\cdot$ is defined by Eq.~(\ref{Zinbiel axiom}) and the Lie bracket $[-,-]$ is defined by Eq.~(\ref{pre-Lie axiom}).
    \end{enumerate}
    \end{cor}

\begin{pro}
Let $(A,\cdot,[-,-])$ be a Poisson algebra. Suppose that $\mathcal{B}$ is a nondegenerate skew-symmetric bilinear form on $A$ such that it is a Connes cocycle on $(A,\cdot)$ and a symplectic form on $(A,[-,-])$. Then there is a compatible pre-Poisson algebra $(A,\star,\circ)$ in which $\star$ and $\circ$ are respectively defined by Eqs.~(\ref{bilinear form Zinbiel}) and (\ref{bilinear form pre-Lie}). Conversely, let $(A,\star,\circ)$ be a pre-Poisson algebra  and the sub-adjacent Poisson algebra be $(A,\cdot,[-,-])$. Then there is a Poisson algebra $A\ltimes_{-\mathcal{L}^{*}_{\star},\mathcal{L}^{*}_{\circ}}A^{*}$, and the natural nondegenerate skew-symmetric bilinear form $\mathcal{B}_{p}$ defined by Eq.~(\ref{B_{p}})
is a Connes cocycle on the commutative associative algebra $A\ltimes_{-\mathcal{L}^{*}_{\star}}A^{*}$ and a symplectic form on the Lie algebra $A\ltimes_{\mathcal{L}^{*}_{\circ}}A^{*}$.
    \end{pro}
\begin{proof}
It is similar to the proof of Proposition \ref{bf PCP}.
\end{proof}

\subsection{PZA algebras}

\begin{defi}
    A \textbf{PZA algebra} is a triple $(A,\star,\circ)$, such that $(A,\star)$ is a Zinbiel algebra, $(A,\circ)$ is an anti-pre-Lie algebra, and Eq.~(\ref{pre-Poisson 1}) and the following equations hold:
    \begin{equation}\label{PZA2}
        (x\circ y-y\circ x)\star z=-x\circ(y\star z)+y\star(x\circ z),
    \end{equation}
\begin{equation}\label{PZA3}
    z\circ(x\star y+y\star x)+z\star(x\circ y+y\circ x)-y\star(z\circ x)-x\star(z\circ y)-x\circ(z\star y)-y\circ(z\star x)=0,
\end{equation}
for all $x,y,z\in A$.
\end{defi}

\begin{pro}
    Let $(A,\cdot,[-,-])$ be a Poisson algebra, $(A,\star)$ be a compatible Zinbiel algebra of $(A,\cdot)$  and $(A,\circ)$ be a compatible anti-pre-Lie algebra of $(A,[-,-])$. If $(\mathcal{L}_{\star},-\mathcal{L}_{\circ},A)$ is a representation of $(A,\cdot,[-,-])$, then $(A,\star,\circ)$ is a PZA algebra. Conversely, let  $(A,\star,\circ)$ be a PZA algebra, $(A,\cdot)$ be the sub-adjacent commutative associative algebra of $(A,\star)$  and $(A,[-,-])$ be the sub-adjacent Lie algebra of $(A,\circ)$. Then $(A,\cdot,[-,-])$ is a Poisson algebra with a representation $(\mathcal{L}_{\star},-\mathcal{L}_{\circ},A)$. In this case, we say $(A,\cdot,[-,-])$ is the \textbf{sub-adjacent Poisson algebra} of $(A,\star,\circ)$, and $(A,\star,\circ)$ is a \textbf{compatible   PZA algebra} of $(A,\cdot,[-,-])$.
\end{pro}
\begin{proof}
   Since $(\mathcal{L}_{\star},-\mathcal{L}_{\circ},A)$ is a representation of $(A,\cdot,[-,-])$, we get Eqs.~(\ref{pre-Poisson 1}) and (\ref{PZA2}).
    Thus for all $x,y,z\in A$, we have
    \begin{eqnarray*}
        0&&=[z,x\cdot y]+[x,z]\cdot y+[y,z]\cdot x\\
        &&=z\circ(x\cdot y)-(x\cdot y)\circ z+[x,z]\star y+y\star [x,z]+[y,z]\star x+x\star[y,z]\\
        &&\overset{(\ref{PZA2})}{=}z\circ(x\cdot y)-(x\cdot y)\circ z+z\star(x\circ y)-x\circ(z\star y)+y\star(x\circ z)-y\star(z\circ x)\\
        &&\ \ \ \ +z\star(y\circ x)-y\circ(z\star x)+x\star(y\circ z)-x\star(z\circ y)\\
        &&\overset{(\ref{pre-Poisson 1})}{=}z\circ(x\star y)+z\circ(y\star x) +z\star(x\circ y)-x\circ(z\star y)\\
        &&\ \ \ \ -y\star(z\circ x)+z\star(y\circ x)-y\circ(z\star x) -x\star(z\circ y).
    \end{eqnarray*}
Hence Eq.~(\ref{PZA3}) holds and thus $(A,\star,\circ)$ is a PZA algebra. The converse part is proved similarly.
\end{proof}

Hence we get the following conclusion.

\begin{cor}
Let $A$ be a vector space with two bilinear operations $\star,\circ:A\otimes A\rightarrow A$. Then the following conditions are equivalent:
\begin{enumerate}
    \item $(A,\star,\circ)$ is a  PZA algebra.
    \item The triple $(A,\cdot,[-,-])$  is a Poisson algebra with a representation $(\mathcal{L}_{\star},-\mathcal{L}_{\circ},A)$, where $\cdot$ and $[-,-]$ are respectively defined by Eqs.~(\ref{comm op}) and (\ref{sub-adj}).
    \item There is a Poisson algebra structure on $A\oplus A$ in which the commutative associative product $\cdot$ is defined by Eq.~(\ref{Zinbiel axiom}) and the Lie bracket $[-,-]$ is defined by Eq.~(\ref{anti-pre-Lie axiom}).
\end{enumerate}
    \end{cor}

\subsection{PAL algebras}

\begin{defi}
    A \textbf{PAL algebra} is a triple $(A,\star,[-,-])$, such that $(A,\star)$ is an anti-Zinbiel algebra, $(A,[-,-])$ is a Lie algebra, and Eq.~(\ref{ZLP2}) and the following equation hold:
    \begin{equation}\label{PAL1}
        [z,x\star y+y\star x]=x\star[y,z]+y\star[x,z],\;\;\forall x,y,z\in A.
    \end{equation}
\end{defi}

\begin{pro}\label{pro:PAL}
    Let $(A,\cdot,[-,-])$ be a Poisson algebra  and $(A,\star)$ be a compatible anti-Zinbiel algebra of $(A,\cdot)$. If $(-\mathcal{L}_{\star},\mathrm{ad},A)$ is a representation of $(A,\cdot,[-,-])$, then $(A,\star,[-,-])$ is a PAL algebra. Conversely,
    let $(A,\star,[-,-])$ be a PAL algebra and $(A,\cdot)$ be the sub-adjacent commutative associative algebra of $(A,\star)$. Then $(A,\cdot,[-,-])$ is a Poisson algebra with a representation $(-\mathcal{L}_{\star},\mathrm{ad},A)$. In this case, we say $(A,\cdot,[-,-])$ is the \textbf{sub-adjacent Poisson algebra} of $(A,\star,[-,-])$, and $(A,\star,[-,-])$ is a \textbf{compatible PAL algebra} of $(A,\cdot,[-,-])$.
\end{pro}
\begin{proof}
    It is similar to the proof of Proposition \ref{pro:PZL}.
\end{proof}

Hence we get the following conclusion.

\begin{cor}
    Let $A$ be a vector space  with two bilinear operations $\star,[-,-]:A\otimes A\rightarrow A$. Then the following conditions are equivalent:
    \begin{enumerate}
        \item $(A,\star,[-,-])$ is a PAL algebra.
        \item The triple $(A,\cdot,[-,-])$ is a Poisson algebra with a representation $(-\mathcal{L}_{\star},\mathrm{ad},A)$, where $\cdot$ is defined by Eq.~(\ref{comm op}).
        \item There is a Poisson algebra structure on $A\oplus A$ in which the commutative associative product $\cdot$ is defined by
        Eq.~(\ref{anti-Zinbiel axiom}) and the Lie bracket $[-,-]$ is defined by  Eq.~(\ref{Lie axiom}).
    \end{enumerate}
    \end{cor}

\begin{pro}
    Let $(A,\cdot,[-,-])$ be a Poisson algebra. Suppose that $\mathcal{B}$ is a nondegenerate symmetric bilinear form on $A$ such that it is a commutative Connes cocycle on $(A,\cdot)$ and invariant on $(A,[-,-])$. Then there is a compatible PAL algebra $(A,\star,[-,-])$ in which $\star$ is defined by Eq.~(\ref{bilinear form anti-Zinbiel}). Conversely, let $(A,\star,[-,-])$ be a PAL algebra  and the sub-adjacent Poisson algebra be $(A,\cdot,[-,-])$. Then there is a Poisson algebra $A\ltimes_{\mathcal{L}^{*}_{\star},\mathrm{ad}^{*}}A^{*}$, and the natural nondegenerate  symmetric bilinear form $\mathcal{B}_{d}$ defined by Eq.~(\ref{B_{d}})
    is a commutative Connes cocycle on the commutative associative algebra $A\ltimes_{\mathcal{L}^{*}_{\star}}A^{*}$ and invariant on the Lie algebra $A\ltimes_{\mathrm{ad}^{*}}A^{*}$.
    \end{pro}
\begin{proof}
It is similar to the proof of Proposition \ref{bf PCP}.
\end{proof}

\subsection{PAP algebras}

\begin{defi}
    A \textbf{PAP algebra} is a triple $(A,\star,\circ)$, such that $(A,\star)$ is an anti-Zinbiel algebra, $(A,\circ)$ is a pre-Lie algebra, and Eq.~(\ref{pre-Poisson 2}) and the following equations hold:
    \begin{equation}\label{PAP}
        x\star(y\circ z)+y\star(x\circ z)=0,
    \end{equation}
\begin{equation}\label{PAP3}
    (x\star y)\circ z+(y\star x)\circ z=0,
\end{equation}
for all $x,y,z\in A$.
\end{defi}

\begin{pro}
    Let $(A,\cdot,[-,-])$ be a Poisson algebra, $(A,\star)$ be a compatible anti-Zinbiel algebra of $(A,\cdot)$  and $(A,\circ)$ be a compatible pre-Lie algebra of $(A,[-,-])$. If $(-\mathcal{L}_{\star},\mathcal{L}_{\circ}$,
    $A)$ is a representation of $(A,\cdot,[-,-])$, then $(A,\star,\circ)$ is a PAP algebra. Conversely, let $(A,\star,\circ)$ be a PAP algebra, $(A,\cdot)$ be the sub-adjacent commutative associative algebra  and $(A,[-,-])$ be the sub-adjacent Lie algebra of $(A,\circ)$. Then $(A,\cdot,[-,-])$ is a Poisson algebra with a representation $(-\mathcal{L}_{\star},\mathcal{L}_{\circ},A)$.
    In this case, we say $(A,\cdot,[-,-])$ is the \textbf{sub-adjacent Poisson algebra} of $(A,\star,\circ)$, and $(A,\star,\circ)$ is a \textbf{compatible PAP algebra} of $(A,\cdot,[-,-])$.
\end{pro}
\begin{proof}
  Since $(-\mathcal{L}_{\star},\mathcal{L}_{\circ},A)$ is a representation of $(A,\cdot,[-,-])$, we get Eqs.~(\ref{pre-Poisson 2}) and the following equation:
   \begin{equation}\label{PAP2}
    (x\star y+y\star x)\circ z=-x\star(y\circ z)-y\star(x\circ z), \;\;\forall x,y,z\in A.
   \end{equation}
   Thus for all $x,y,z\in A$, we have
    \begin{eqnarray*}
        0&&=[z,x\cdot y]+[x,z]\cdot y+[y,z]\cdot x\\
        &&=z\circ(x\star y)+z\circ(y\star x)-(x\star y)\circ z-(y\star x)\circ z+(x\circ z-z\circ x)\star y\\
        &&\ \ \ \ +y\star(x\circ z-z\circ x)+(y\circ z-z\circ y)\star x+x\star(y\circ z-z\circ y)\\
        &&\overset{(\ref{pre-Poisson 2}),(\ref{PAP2})}{=}-2(x\star y)\circ z-2(y\star x)\circ z.
    \end{eqnarray*}
Hence Eq.~(\ref{PAP3}) holds, and by Eq.~(\ref{PAP2}), Eq.~(\ref{PAP}) holds. Thus $(A,\star,\circ)$ is a PAP algebra. The converse part is proved similarly.
\end{proof}

Hence we get the following conclusion.

\begin{cor}
Let $A$ be a vector space with two bilinear operations $\star,\circ:A\otimes A\rightarrow A$. Then the following conditions are equivalent:
\begin{enumerate}
    \item $(A,\star,\circ)$ is a  PAP algebra.
    \item The triple $(A,\cdot,[-,-])$ is a Poisson algebra with a representation $(-\mathcal{L}_{\star},\mathcal{L}_{\circ},A)$, where $\cdot$ and $[-,-]$ are respectively defined by Eqs.~(\ref{comm op}) and (\ref{sub-adj}).
    \item There is a Poisson algebra structure on $A\oplus A$ in which the commutative associative product $\cdot$ is defined by Eq.~(\ref{anti-Zinbiel axiom}) and the Lie bracket $[-,-]$ is defined by Eq.~(\ref{pre-Lie axiom}).
\end{enumerate}
    \end{cor}

\subsection{PAA algebras}

\begin{defi}
    A \textbf{PAA algebra} is a triple $(A,\star,\circ)$, such that $(A,\star)$ is an anti-Zinbiel algebra, $(A,\circ)$ is an anti-pre-Lie algebra, and Eqs.~(\ref{PZA2}),(\ref{PAP2}) and the following equation   hold:
\begin{equation}\label{PAA3}
z\circ(x\star y+y\star x)-(x\star y+y\star x)\circ z-x\star(z\circ y)-y\star(z\circ x)=0,
\end{equation}
for all $x,y,z\in A$.
\end{defi}

\begin{pro}
    Let $(A,\cdot,[-,-])$ be a Poisson algebra, $(A,\star)$ be a compatible anti-Zinbiel algebra of $(A,\cdot)$  and $(A,\circ)$ be a compatible anti-pre-Lie algebra of $(A,[-,-])$. If $(-\mathcal{L}_{\star},-\mathcal{L}_{\circ},A)$ is a representation of $(A,\cdot,[-,-])$, then $(A,\star,\circ)$ is a PAA algebra. Conversely, let $(A,\star,\circ)$ be a PAA algebra, $(A,\cdot)$ be the sub-adjacent commutative associative algebra of $(A,\star)$  and $(A,[-,-])$ be the sub-adjacent Lie algebra of $(A,\circ)$. Then $(A,\cdot,[-,-])$ is a Poisson algebra with a representation $(-\mathcal{L}_{\star},-\mathcal{L}_{\circ},A)$.
    In this case, we say $(A,\cdot,[-,-])$ is the \textbf{sub-adjacent Poisson algebra} of $(A,\star,\circ)$, and $(A,\star,\circ)$ is a \textbf{compatible PAA algebra} of $(A,\cdot,[-,-])$.
\end{pro}
\begin{proof}
   Since $(-\mathcal{L}_{\star},-\mathcal{L}_{\circ},A)$ is a representation of $(A,\cdot,[-,-])$, we get Eqs.~(\ref{PZA2}) and (\ref{PAP2}). Thus  for all $x,y,z\in A$, we have
    \begin{eqnarray*}
        0&&=[z,x\cdot y]+[x,z]\cdot y+[y,z]\cdot x\\
        &&=z\circ(x\star y)+z\circ(y\star x)-(x\star y)\circ z-(y\star x)\circ z+(x\circ z-z\circ x)\star y\\
        &&\ \ \ \ +y\star(x\circ z-z\circ x)+(y\circ z-z\circ y)\star x+x\star(y\circ z-z\circ y)\\
        &&\overset{(\ref{PZA2}),(\ref{PAP2})}{=}2(z\circ(x\star y+y\star x)-(x\star y+y\star x)\circ z-x\star(z\circ y)-y\star(z\circ x)).
    \end{eqnarray*}
Hence Eq.~(\ref{PAA3}) holds, and thus $(A,\star,\circ)$ is a PAA algebra.
The converse part is proved similarly.
\end{proof}

Hence we get the following conclusion.

\begin{cor}
    Let $A$ be a vector space with two bilinear operations $\star,\circ:A\otimes A\rightarrow A$. Then the following conditions are equivalent:
    \begin{enumerate}
        \item $(A,\star,\circ)$ is a   PAA algebra.
        \item The triple $(A,\cdot,[-,-])$ is a Poisson algebra with a representation $(-\mathcal{L}_{\star},-\mathcal{L}_{\circ},A)$, where $\cdot$ and $[-,-]$ are respectively defined by Eqs.~(\ref{comm op}) and (\ref{sub-adj}).
        \item There is a Poisson algebra structure on $A\oplus A$ in which the commutative associative product $\cdot$ is defined by Eq.~(\ref{anti-Zinbiel axiom}) and the Lie bracket $[-,-]$ is defined by Eq.~(\ref{anti-pre-Lie axiom}).
    \end{enumerate}
\end{cor}

\begin{pro}
    Let $(A,\cdot,[-,-])$ be a Poisson algebra. Suppose that $\mathcal{B}$ is a nondegenerate symmetric bilinear form on $A$ such that it is a commutative Connes cocycle on $(A,\cdot)$ and a commutative 2-cocycle on $(A,[-,-])$. Then there is a compatible PAA algebra $(A,\star,\circ)$ in which $\star$ and $\circ$ are respectively defined by Eqs.~(\ref{bilinear form anti-Zinbiel}) and (\ref{bilinear form anti-pre-Lie}). Conversely, let $(A,\star,\circ)$ be a   PAA algebra and the sub-adjacent Poisson algebra be $(A,\cdot,[-,-])$. Then there is a Poisson algebra $A\ltimes_{\mathcal{L}^{*}_{\star},-\mathcal{L}^{*}_{\circ}}A^{*}$, and the natural nondegenerate symmetric bilinear form $\mathcal{B}_{d}$ defined by Eq.~(\ref{B_{d}})
    is a commutative Connes cocycle on the commutative associative algebra $A\ltimes_{\mathcal{L}^{*}_{\star}}A^{*}$ and a commutative 2-cocycle on the Lie algebra $A\ltimes_{-\mathcal{L}^{*}_{\circ}}A^{*}$.
\end{pro}
\begin{proof}
It is similar to the proof of Proposition \ref{bf PCP}.
\end{proof}


\subsection{Summary}\

We summarize some facts on the 8 algebraic structures in the previous subsections respectively corresponding to the mixed splittings of operations of Poisson algebras in Table 1.
\begin{small}
\begin{table}[!t]
    \caption{Splittings of Poisson algebras}
    %
    %
    \begin{tabular}{p{2.5cm}p{2.5cm}p{2.5cm}p{2.5cm}p{5cm}}
        \hline
        Algebras & Notations & Representations

            of Poisson

            algebras  on

            the spaces

            themselves & Representations

            of Poisson

            algebras  on

            the dual spaces & {Corresponding nondegenerate

                bilinear forms

            on Poisson algebras}  \\
         PCP   & $(A,\cdot,\circ)$  & $(\mathcal{L}_{\cdot},\mathcal{L}_{\circ},A)$ & $(-\mathcal{L}^{*}_{\cdot},\mathcal{L}^{*}_{\circ},A^{*})$  & - \\

        PCA  & $(A,\cdot,\circ)$ & $(\mathcal{L}_{\cdot},-\mathcal{L}_{\circ},A)$ & $(-\mathcal{L}^{*}_{\cdot},-\mathcal{L}^{*}_{\circ},A^{*})$  & invariant,

        commutative 2-cocycle  \\

        PZL & $(A,\star,[-,-])$ & $(\mathcal{L}_{\star},\mathrm{ad},A)$ & $(-\mathcal{L}^{*}_{\star},\mathrm{ad}^{*},A^{*})$ & - \\

        pre-Poisson & $(A,\star,\circ)$ & $(\mathcal{L}_{\star},\mathcal{L}_{\circ},A)$ & $(-\mathcal{L}^{*}_{\star},\mathcal{L}^{*}_{\circ},A^{*})$ & Connes cocycle,

        symplectic form\\

        PZA & $(A,\star,\circ)$ & $(\mathcal{L}_{\star},-\mathcal{L}_{\circ},A)$ & $(-\mathcal{L}^{*}_{\star},-\mathcal{L}^{*}_{\circ},A^{*})$ & -\\

        PAL & $(A,\star,[-,-])$ & $(-\mathcal{L}_{\star},\mathrm{ad},A)$ &$(\mathcal{L}^{*}_{\star},\mathrm{ad}^{*},A^{*})$ & commutative Connes cocycle,

        invariant\\

        PAP & $(A,\star,\circ)$ & $(-\mathcal{L}_{\star},\mathcal{L}_{\circ},A)$ & $(\mathcal{L}^{*}_{\star},\mathcal{L}^{*}_{\circ},A^{*})$ & -\\

        PAA & $(A,\star,\circ)$ & $(-\mathcal{L}_{\star},-\mathcal{L}_{\circ},A)$ & $(\mathcal{L}^{*}_{\star},-\mathcal{L}^{*}_{\circ},A^{*})$ & commutative Connes cocycle,

        commutative 2-cocycle\\
       \hline
    \end{tabular}
\end{table}
\end{small}

\section{Mixed splittings of operations of transposed Poisson algebras in terms of representations on the spaces themselves
and related algebraic structures}\label{sec:3}\

We introduce 8 algebraic structures respectively corresponding to the mixed splitting of the commutative associative
 products and Lie brackets of transposed Poisson algebras interlacedly in three manners: the classical splitting, the second splitting and the un-splitting, in terms of representations of transposed Poisson algebras on the spaces themselves.

\begin{defi}\cite{LB2023}
    A \textbf{representation of a transposed Poisson algebra} $(A,\cdot,[-,-])$ is a triple $(\mu,\rho,V)$, such that $(\mu,V)$ is a representation of the commutative associative algebra $(A,\cdot)$, $(\rho,V)$ is a representation of the Lie algebra $(A,[-,-])$, and the following compatible conditions hold:
    \begin{equation}\label{eq:defi:TPA REP1}
    2\mu(x)\rho(y)=\rho(x\cdot y)+\rho(y)\mu(x),
    \end{equation}
    \begin{equation}\label{eq:defi:TPA REP2}
    2\mu([x,y])=\rho(x)\mu(y)-\rho(y)\mu(x),
    \end{equation}
 for all $x,y\in A$.
\end{defi}

Let $(A,\cdot,[-,-])$ be a transposed Poisson algebra. Then $(\mathcal{L}_{\cdot},\mathrm{ad},A)$ is a representation of $(A,\cdot,[-,-])$, called the \textbf{adjoint representation} of $(A,\cdot,[-,-])$.
Moreover, $(\mu,\rho,V)$ is a representation of a transposed Poisson algebra $(A,\cdot,[-,-])$ if and only if the direct sum $A\oplus V$ of vector spaces is a (\textbf{semi-direct product}) transposed Poisson algebra by defining the multiplications on $A\oplus V$ by Eqs.~(\ref{eq:SDASSO}) and (\ref{eq:SDLie})   respectively. We denote it by $A\ltimes_{\mu,\rho}V$.

Unlike the case of Poisson algebras in Proposition \ref{pro:Poisson rep}, for a representation $(\mu,\rho,V)$ of a transposed Poisson algebra $(A,\cdot,[-,-])$, $(-\mu^{*},\rho^{*},V^{*})$ is not necessarily a representation of $(A,\cdot,[-,-])$ (see Proposition \ref{pro:dual rep TPA}). Thus for transposed Poisson algebras, we shall divide into two cases according to the representations of transposed Poisson algebras on the spaces themselves
and the dual spaces  respectively.


Next we introduce 8 algebraic structures in the rest of this section  corresponding to the mixed splitting of  the commutative associative
 products and Lie brackets of transposed Poisson algebras interlacedly in three manners: the classical splitting, the second splitting and the un-splitting, in terms of representations of transposed Poisson algebras on the spaces themselves,
 whereas another 8 algebraic structures are introduced in the next section in terms of representations of transposed Poisson algebras on the dual spaces.
Before we introduce these various algebraic structures, we give the following ``principle" to name them.
\begin{enumerate}
\item Every such algebraic structure in this section and the next section  is named by 4 capital letters.
\item The first letter is unified to be ``T" since these algebras are related to transposed Poisson algebras.
\item The second letter denotes the operation corresponding to the splitting of the commutative associative products. Explicitly,  the capital letters ``C", ``Z" and ``A" respectively denote the operations of
 commutative associative algebras, Zinbiel algebras and anti-Zinbiel algebras, corresponding to the un-splitting, the classical splitting and the second splitting.
\item The third letter denotes the operation corresponding to the splitting of the Lie brackets. Explicitly,  the capital letters ``L", ``P" and ``A" respectively denote the operations of
 Lie algebras, pre-Lie algebras and anti-pre-Lie algebras, corresponding to the un-splitting, the classical splitting and the second splitting.
\item The last letter is ``O" in the case of the representations of transposed Poisson algebras on the spaces themselves and ``D" in the case of the representations of transposed Poisson algebras on the dual spaces.
\end{enumerate}

\subsection{TCPO algebras}

\begin{defi}
A \textbf{TCPO algebra} is a triple $(A,\cdot,\circ)$, where $(A,\cdot)$ is a commutative associative algebra, $(A,\circ)$ is a pre-Lie algebra, and
\begin{equation}\label{TCPO1}
    2x\cdot (y\circ z)=(x\cdot y)\circ z+y\circ(x\cdot z),
\end{equation}
\begin{equation}\label{TCPO2}
    2z\cdot(x\circ y-y\circ x)=x\circ(z\cdot y)-y\circ(z\cdot x),
\end{equation}
for all $x,y,z\in A.$
\end{defi}

\begin{pro}
    Let $(A,\cdot,[-,-])$ be a transposed Poisson algebra and  $(A,\circ)$ be a compatible pre-Lie algebra of $(A,[-,-])$. If $(\mathcal{L}_{\cdot},\mathcal{L}_{\circ},A)$ is a representation of $(A,\cdot,[-,-])$, then $(A,\cdot,\circ)$ is a TCPO algebra.
Conversely, let $(A,\cdot,\circ)$ be a TCPO algebra and $(A,[-,-])$ be the sub-adjacent Lie algebra of $(A,\circ)$. Then $(A,\cdot,[-,-])$ is a transposed Poisson algebra with a representation $(\mathcal{L}_{\cdot},\mathcal{L}_{\circ},A)$. In this case, we say $(A,\cdot,[-,-])$ is the \textbf{sub-adjacent transposed Poisson algebra} of $(A,\cdot,\circ)$, and $(A,\cdot,\circ)$ is a \textbf{compatible TCPO algebra} of $(A,\cdot,[-,-])$.
\end{pro}
\begin{proof}
    We only prove the latter.
Let $x,y,z\in A$. We have
\begin{equation}\label{TCPO3}
    (x\cdot y)\circ z-(x\cdot z)\circ y
    \overset{(\ref{TCPO1})}{=}
    2x\cdot(y\circ z)-y\circ(x\cdot z)-2x\cdot(z\circ y)+z\circ(x\cdot y)\overset{(\ref{TCPO2})}{=}0.
\end{equation}
Then
\begin{eqnarray*}
    &&2[x,y]\cdot z-[z\cdot x,y]-[x,z\cdot y]\\
    &&\overset{(\ref{TCPO2})}{=}x\circ(y\cdot z)-y\circ(x\cdot z)+y\circ(z\cdot x)-(z\cdot x)\circ y-x\circ(z\cdot y)+(z\cdot y)\circ x\\
    &&=(z\cdot y)\circ x-(z\cdot x)\circ y\overset{(\ref{TCPO3})}{=}0.
\end{eqnarray*}
Thus $(A,\cdot,[-,-])$ is a transposed Poisson algebra, and by Eqs.~(\ref{TCPO1})-(\ref{TCPO2}), $(\mathcal{L}_{\cdot},\mathcal{L}_{\circ},A)$ is a representation of $(A,\cdot,[-,-])$.
\end{proof}

Hence we get the following conclusion.

\begin{cor}
Let $A$ be a vector space with  two bilinear operations $\cdot,\circ:A\otimes A\rightarrow A$. Then the following conditions are equivalent:
\begin{enumerate}
    \item $(A,\cdot,\circ)$ is a TCPO algebra.
    \item The triple $(A,\cdot,[-,-])$ is a transposed Poisson algebra with a representation $(\mathcal{L}_{\cdot},\mathcal{L}_{\circ},A)$, where $[-,-]$ is defined by Eq.~(\ref{sub-adj}).
    \item There is a transposed Poisson algebra structure on $A\oplus A$ in which the commutative associative product $\cdot$ is defined by Eq.~(\ref{comm asso axiom}) and the Lie bracket $[-,-]$ is defined by Eq.~(\ref{pre-Lie axiom}).
\end{enumerate}
    \end{cor}

\subsection{TCAO algebras}

\begin{defi}
    A \textbf{TCAO algebra} is a triple $(A,\cdot,\circ)$, where $(A,\cdot)$ is a commutative associative algebra, $(A,\circ)$ is an anti-pre-Lie algebra, and Eq.~(\ref{TCPO1}) and the following equation hold:
    \begin{equation}\label{TCAO2}
        2z\cdot(x\circ y-y\circ x)=-x\circ(z\cdot y)+y\circ(z\cdot x),\;\;\forall x,y,z\in A.
    \end{equation}
\end{defi}

\begin{pro}
    Let $(A,\cdot,[-,-])$ be a transposed Poisson algebra and $(A,\circ)$ be a compatible anti-pre-Lie algebra of $(A,[-,-])$. If $(\mathcal{L}_{\cdot},-\mathcal{L}_{\circ},A)$ is a representation of $(A,\cdot,[-,-])$, then $(A,\cdot,\circ)$ is a TCAO algebra.
    Conversely, let $(A,\cdot,\circ)$ be a TCAO algebra  and $(A,[-,-])$ be the sub-adjacent Lie algebra of $(A,\circ)$. Then $(A,\cdot,[-,-])$ is a transposed Poisson algebra with a representation $(\mathcal{L}_{\cdot},-\mathcal{L}_{\circ},A)$. In this case, we say $(A,\cdot,[-,-])$ is the \textbf{sub-adjacent transposed Poisson algebra} of $(A,\cdot,\circ)$, and $(A,\cdot,\circ)$ is a \textbf{compatible TCAO algebra} of $(A,\cdot,[-,-])$.
\end{pro}
\begin{proof}
    We only prove the latter.
    Let $x,y,z\in A$. We have
    \begin{equation}\label{TCAO3}
        (x\cdot y)\circ z-(x\cdot z)\circ y
        \overset{(\ref{TCPO1})}{=}
        2x\cdot(y\circ z)-y\circ(x\cdot z)-2x\cdot(z\circ y)+z\circ(x\cdot y)\overset{(\ref{TCAO2})}{=}2z\circ(x\cdot y)-2y\circ(x\cdot z).
    \end{equation}
    Then
    \begin{eqnarray*}
        &&2[x,y]\cdot z-[z\cdot x,y]-[x,z\cdot y]\\
        &&\overset{(\ref{TCAO2})}{=}y\circ(x\cdot z)-x\circ(y\cdot z)+y\circ(z\cdot x)-(z\cdot x)\circ y-x\circ(z\cdot y)+(z\cdot y)\circ x\\
        &&=2y\circ(z\cdot x)-2x\circ(y\cdot z)-(z\cdot x)\circ y+(z\cdot y)\circ x\overset{(\ref{TCAO3})}{=}0.
    \end{eqnarray*}
    Thus $(A,\cdot,[-,-])$ is a transposed Poisson algebra, and by Eqs.~(\ref{TCPO1}) and (\ref{TCAO2}), $(\mathcal{L}_{\cdot},-\mathcal{L}_{\circ},A)$ is a representation of $(A,\cdot,[-,-])$.
\end{proof}

Hence we get the following conclusion.

\begin{cor}
    Let $A$ be a vector space with  two bilinear operations $\cdot,\circ:A\otimes A\rightarrow A$. Then the following conditions are equivalent:
    \begin{enumerate}
        \item $(A,\cdot,\circ)$ is a TCAO algebra.
        \item The triple $(A,\cdot,[-,-])$ is a transposed Poisson algebra with a representation $(\mathcal{L}_{\cdot},-\mathcal{L}_{\circ}$,
        $A)$, where $[-,-]$ is defined by Eq.~(\ref{sub-adj}).
        \item There is a transposed Poisson algebra structure on $A\oplus A$ in which the commutative associative product $\cdot$ is defined by Eq.~(\ref{comm asso axiom}) and the Lie bracket $[-,-]$ is defined by Eq.~(\ref{anti-pre-Lie axiom}).
    \end{enumerate}
    \end{cor}

\subsection{TZLO algebras}

\begin{defi}
    A \textbf{TZLO algebra} is a triple $(A,\star,[-,-])$, such that $(A,\star)$ is a Zinbiel algebra, $(A,[-,-])$ is a Lie algebra, and the following equation  holds:
    \begin{equation}\label{TZLO}
    x\star[y,z]=[x,y]\star z=[x,y\star z]=0,\;\;\forall x,y,z\in A.
    \end{equation}
\end{defi}

\begin{pro}
    Let $(A,\cdot,[-,-])$ be a transposed Poisson algebra and $(A,\star)$ be a compatible Zinbiel algebra of $(A,\cdot)$. If $(\mathcal{L}_{\star},\mathrm{ad},A)$ is a representation of $(A,\cdot,[-,-])$, then    $(A,\star,[-,-])$ is a TZLO algebra. Conversely, let $(A,\star,[-,-])$ be a TZLO algebra and $(A,\cdot)$ be the sub-adjacent commutative associative algebra of $(A,\star)$. Then $(A,\cdot,[-,-])$ is a transposed Poisson algebra with a representation $(\mathcal{L}_{\star},\mathrm{ad},A)$. In this case, we say $(A,\cdot,[-,-])$ is the \textbf{sub-adjacent transposed Poisson algebra} of $(A,\star,[-,-])$, and $(A,\star,[-,-])$ is a \textbf{compatible   TZLO algebra} of $(A,\cdot,[-,-])$.
\end{pro}
\begin{proof}
    Since $(\mathcal{L}_{\star},\mathrm{ad},A)$ is a representation of $(A,\cdot,[-,-])$, the following equations hold:
    \begin{equation}\label{TZLO1}
        2x\star[y,z]=[x\star y+y\star x,z]+[y,x\star z],
    \end{equation}
    \begin{equation}\label{TZLO2}
        2[x,y]\star z=[x,y\star z]-[y,x\star z],
    \end{equation}
    for all $x,y,z\in A$.
    Thus for all $x,y,z\in A$, we have
    \begin{eqnarray*}
        0&&=2[x,y]\cdot z-[z\cdot x,y]-[x,z\cdot y]\\
        &&\overset{(\ref{TZLO1}),(\ref{TZLO2})}{=}[x,y\star z]-[y,x\star z]+[z\cdot x,y]+[x,z\star y]-[z\cdot x,y]-[x,z\cdot y]\\
        &&=-[y,x\star z].
    \end{eqnarray*}
By Eqs.~(\ref{TZLO1})-(\ref{TZLO2}) again, we get Eq.~(\ref{TZLO}). Hence $(A,\star,[-,-])$ is a TZLO algebra. The converse part is proved similarly.
\end{proof}

\begin{rmk}
    Let $(A,\star,[-,-])$ be a TZLO algebra. Then the sub-adjacent transposed Poisson algebra $(A,\cdot,[-,-])$ is trivial in the sense that
    \begin{equation}\label{trivial TPA}
        [x,y\cdot z]=x\cdot[y,z]=0,\;\;\forall x,y,z\in A.
    \end{equation}
Note that in this case, it is also a Poisson algebra \cite{Bai2020}.
    \end{rmk}

Moreover we get the following conclusion.

\begin{cor}
    Let $A$ be a vector space with  two bilinear operations $\star,[-,-]:A\otimes A\rightarrow A$. Then the following conditions are equivalent:
    \begin{enumerate}
        \item $(A,\star,[-,-])$ is a   TZLO algebra.
        \item The triple $(A,\cdot,[-,-])$ is a transposed Poisson algebra with a representation $(\mathcal{L}_{\star},\mathrm{ad}$,
        $A)$, where $\cdot$ is defined by Eq.~(\ref{comm op}).
        \item There is a transposed Poisson algebra structure on $A\oplus A$ in which the commutative associative product $\cdot$ is defined by Eq.~(\ref{Zinbiel axiom}) and the Lie bracket $[-,-]$ is defined by Eq.~(\ref{Lie axiom}).
    \end{enumerate}
\end{cor}

\subsection{TZPO algebras}

\begin{defi}
    A \textbf{TZPO algebra} is a triple $(A,\star,\circ)$, such that $(A,\star)$ is a Zinbiel algebra, $(A,\circ)$ is a pre-Lie algebra, and the following equations hold:
    \begin{equation}\label{PTPA1}
        2x\star(y\circ z)=(x\star y+y\star x)\circ z+y\circ(x\star z),
    \end{equation}
    \begin{equation}\label{PTPA2}
        2(x\circ y-y\circ x)\star z=x\circ(y\star z)-y\circ(x\star z),
    \end{equation}
    for all $x,y,z\in A$.
\end{defi}

\begin{rmk}
     In fact, TZPO algebras might be named as  ``pre-transposed Poisson algebras" since the operad of TZPO algebras is the successor of the operad of transposed Poisson algebras,
     illustrating the classical splitting of operations of transposed Poisson algebras.
     \end{rmk}

\begin{pro}
    Let $(A,\cdot,[-,-])$ be a transposed Poisson algebra, $(A,\star)$ be a compatible Zinbiel algebra of $(A,\cdot)$ and $(A,\circ)$ be a compatible pre-Lie algebra of $(A,[-,-])$. If $(\mathcal{L}_{\star},\mathcal{L}_{\circ},A)$ is a representation of $(A,\cdot,[-,-])$, then $(A,\star,\circ)$ is a TZPO algebra.
    Conversely,
    let $(A,\star,\circ)$ be a TZPO algebra, $(A,\cdot)$ be the sub-adjacent commutative associative algebra of $(A,\star)$  and $(A,[-,-])$ be the sub-adjacent Lie algebra of $(A,\circ)$. Then $(A,\cdot,[-,-])$ is a transposed Poisson algebra with a representation $(\mathcal{L}_{\star},\mathcal{L}_{\circ},A)$. In this case, we say $(A,\cdot,[-,-])$ is the \textbf{sub-adjacent transposed Poisson algebra} of $(A,\star,\circ)$, and $(A,\star,\circ)$ is a \textbf{compatible TZPO algebra} of $(A,\cdot,[-,-])$.
\end{pro}
\begin{proof}
    We only prove the latter. For all $x,y,z\in A$, we have
    \begin{eqnarray*}
        &&2[x,y]\cdot z-[z\cdot x,y]-[x,z\cdot y]\\
        &&\overset{(\ref{PTPA1}),(\ref{PTPA2})}{=}x\circ(y\star z)-y\circ(x\star z)+(z\cdot x)\circ y+x\circ(z\star y)-(z\cdot y)\circ x\\
        &&\ \ \ \ -y\circ(z\star x)-(z\cdot x)\circ y+y\circ(z\cdot x)-x\circ(z\cdot y)+(z\cdot y)\circ x\\
        &&=0.
    \end{eqnarray*}
Hence $(A,\cdot,[-,-])$ is a transposed Poisson algebra, and by Eqs.~(\ref{PTPA1})-(\ref{PTPA2}), $(\mathcal{L}_{\star},\mathcal{L}_{\circ},A)$ is a representation of $(A,\cdot,[-,-])$.
\end{proof}

\begin{ex}
    Let $(A,\star)$ be a Zinbiel algebra. Suppose $P$ is a derivation of $(A,\star)$, that is, $P$ satisfies
    \begin{equation}
        P(x\star y)=P(x)\star y+x\star P(y),\;\;\forall x,y\in A.
    \end{equation}
Then $(A,\circ)$ is a pre-Lie algebra, where
\begin{equation}
    x\circ y=P(x)\star y-x\star P(y),\;\;\forall x,y\in A.
\end{equation}
Moreover, $(A,\star,\circ)$ is a TZPO algebra. Note that for the sub-adjacent transposed Poisson algebra $(A,\cdot,[-,-])$, where
$\cdot$ and $[-,-]$ are respectively defined by Eqs.~(\ref{comm op}) and (\ref{sub-adj}), the following equation holds:
\begin{equation}\label{eq:Lie from diff comm asso}
    [x,y]=P(x)\cdot y-x\cdot P(y),\;\;\forall x,y\in A.
\end{equation}
    \end{ex}

Moreover we get the following conclusion.

\begin{cor}
    Let $A$ be a vector space with  two bilinear operations $\star,\circ:A\otimes A\rightarrow A$. Then the following conditions are equivalent:
    \begin{enumerate}
        \item $(A,\star,\circ)$ is a TZPO algebra.
        \item The triple $(A,\cdot,[-,-])$  is a transposed Poisson algebra with a representation $(\mathcal{L}_{\star},\mathcal{L}_{\circ}$,
        $A)$, where $\cdot$ and $[-,-]$ are respectively defined by Eqs.~(\ref{comm op}) and (\ref{sub-adj}).
        \item There is a transposed Poisson algebra structure on $A\oplus A$ in which the commutative associative product $\cdot$ is defined by Eq.~(\ref{Zinbiel axiom}) and the Lie bracket $[-,-]$ is defined by Eq.~(\ref{pre-Lie axiom}).
    \end{enumerate}
\end{cor}

\subsection{TZAO algebras}

\begin{defi}
    A \textbf{TZAO algebra} is a triple $(A,\star,\circ)$, such that $(A,\star)$ is a Zinbiel algebra, $(A,\circ)$ is an anti-pre-Lie algebra, and Eq.~(\ref{PTPA1}) and the following equations  hold:
    \begin{equation}\label{TZAO4}
        (x\circ y)\star z-(y\circ x)\star z=0,
    \end{equation}
     \begin{equation}\label{TZAO3}
        x\circ(y\star z)-y\circ(x\star z)=0,
    \end{equation}
    for all $x,y,z\in A$.
\end{defi}

\begin{pro}
    Let $(A,\cdot,[-,-])$ be a transposed Poisson algebra, $(A,\star)$ be a compatible Zinbiel algebra of $(A,\cdot)$ and $(A,\circ)$ be a compatible anti-pre-Lie algebra of $(A,[-,-])$. If $(\mathcal{L}_{\star},-\mathcal{L}_{\circ},A)$ is a representation of $(A,\cdot,[-,-])$, then $(A,\star,\circ)$ is a TZAO algebra.
    Conversely, let $(A,\star,\circ)$ be a TZAO algebra, $(A,\cdot)$ be the sub-adjacent commutative associative algebra of $(A,\star)$ and $(A,[-,-])$ be the sub-adjacent Lie algebra of $(A,\circ)$. Then $(A,\cdot,[-,-])$ is a transposed Poisson algebra with a representation $(\mathcal{L}_{\star},-\mathcal{L}_{\circ},A)$. In this case, we say $(A,\cdot,[-,-])$ is the \textbf{sub-adjacent transposed Poisson algebra} of $(A,\star,\circ)$, and $(A,\star,[-,-])$ is a \textbf{compatible TZAO algebra} of $(A,\cdot,[-,-])$.
\end{pro}
\begin{proof}
   Since $(\mathcal{L}_{\star},-\mathcal{L}_{\circ},A)$ is a representation of $(A,\cdot,[-,-])$, Eq.~(\ref{PTPA1}) and the following equation hold:
     \begin{equation}\label{TZAO2}
        2(x\circ y-y\circ x)\star z=y\circ(x\star z)-x\circ(y\star z),\;\;\forall x,y,z\in A.
    \end{equation}
    Thus for all $x,y,z\in A$, we have
    \begin{eqnarray*}
        0&&=2[x,y]\cdot z-[z\cdot x,y]-[x,z\cdot y]\\
        &&\overset{(\ref{PTPA1}),(\ref{TZAO2})}{=}y\circ(x\star z)-x\circ(y\star z)+(z\cdot x)\circ y+x\circ(z\star y)-(z\cdot y)\circ x\\
        &&\ \ \ \ -y\circ(z\star x)-(z\cdot x)\circ y+y\circ(z\cdot x)-x\circ(z\cdot y)+(z\cdot y)\circ x\\
        &&=2y\circ(x\star z)-2x\circ(y\star z).
    \end{eqnarray*}
Hence Eq.~(\ref{TZAO3}) holds. Substituting  Eq.~(\ref{TZAO3}) into Eq.~(\ref{TZAO2}), we get Eq.~(\ref{TZAO4}). Thus $(A,\star,\circ)$ is a TZAO algebra. The converse part is proved similarly.
\end{proof}

Hence  we get the following conclusion.

\begin{cor}
    Let $A$ be a vector space with  two bilinear operations $\star,\circ:A\otimes A\rightarrow A$. Then the following conditions are equivalent:
    \begin{enumerate}
        \item $(A,\star,\circ)$ is a   TZAO algebra.
        \item The triple $(A,\cdot,[-,-])$ is a transposed Poisson algebra with a representation $(\mathcal{L}_{\star},-\mathcal{L}_{\circ}$,
        $A)$, where $\cdot$ and $[-,-]$ are respectively defined by Eqs.~(\ref{comm op}) and (\ref{sub-adj}).
        \item There is a transposed Poisson algebra structure on $A\oplus A$ in which the commutative associative product $\cdot$ is defined by Eq.~(\ref{Zinbiel axiom}) and the Lie bracket $[-,-]$ is defined by Eq.~(\ref{anti-pre-Lie axiom}).
    \end{enumerate}
\end{cor}

\subsection{TALO algebras}

\begin{defi}
    A \textbf{TALO algebra} is a triple $(A,\star,[-,-])$, such that $(A,\star)$ is an anti-Zinbiel algebra, $(A,[-,-])$ is a Lie algebra, and Eq.~(\ref{TZLO}) holds.
\end{defi}

\begin{pro}
    Let $(A,\cdot,[-,-])$ be a transposed Poisson algebra and $(A,\star)$ be a compatible anti-Zinbiel algebra of $(A,\cdot)$. If $(-\mathcal{L}_{\star},\mathrm{ad},A)$ is a representation of $(A,\cdot,[-,-])$, then $(A,\star,[-,-])$ is a TALO algebra. Conversely, let $(A,\star,[-,-])$ be a TALO algebra and $(A,\cdot)$ be the sub-adjacent commutative associative algebra of $(A,\star)$. Then $(A,\cdot,[-,-])$ is a transposed  Poisson algebra with a representation $(-\mathcal{L}_{\star},\mathrm{ad},A)$. In this case, we say $(A,\cdot,[-,-])$ is the \textbf{sub-adjacent transposed Poisson algebra} of $(A,\star,[-,-])$, and $(A,\star,[-,-])$ is a \textbf{compatible  TALO algebra} of $(A,\cdot,[-,-])$.
\end{pro}
\begin{proof}
    Since $(-\mathcal{L}_{\star},\mathrm{ad},A)$ is a representation of $(A,\cdot,[-,-])$, Eq.~(\ref{TZLO2}) and the following equation hold:
    \begin{equation}\label{TALO1}
        2x\star[y,z]=[z,x\star y+y\star x]+[y,x\star z],\;\;\forall x,y,z\in A.
    \end{equation}
    Thus for all $x,y,z\in A$, we have
    \begin{eqnarray*}
        0&&=2[x,y]\cdot z-[z\cdot x,y]-[x,z\cdot y]\\
        &&\overset{(\ref{TALO1}),(\ref{TZLO2})}{=}[x,y\star z]-[y,x\star z]-[z\cdot x,y]+[x,z\star y]-[z\cdot x,y]-[x,z\cdot y]\\
        &&=[y,x\cdot z+z\star x].
    \end{eqnarray*}
Hence we get
    \begin{equation}\label{TALO4}
        [y,x\cdot z+z\star x]=0,
    \end{equation}
    \begin{equation}\label{TALO5}
        [y,z\cdot x+x\star z]=0.
    \end{equation}
    Adding them together, we get
    \begin{equation}\label{TALO6}
        [y,z\cdot x]=0.
    \end{equation}
    Combining it with Eq.~(\ref{TALO4}), we get
    \begin{equation}
        [x,y\star z]=0.
    \end{equation}
Then by Eqs.~(\ref{TZLO2}) and (\ref{TALO1}), we get Eq.~(\ref{TZLO}).   Thus $(A,\star,[-,-])$ is a TALO algebra. The converse part is proved similarly.
\end{proof}

\begin{rmk}
    For a TALO algebra $(A,\star,[-,-])$, the sub-adjacent transposed Poisson algebra $(A,\cdot,[-,-])$ is also trivial in the sense of Eq.~(\ref{trivial TPA}). Hence in this case, it is a Poisson algebra.
\end{rmk}

Moreover  we get the following conclusion.

\begin{cor}
    Let $A$ be a vector space with  two bilinear operations $\star,\circ:A\otimes A\rightarrow A$. Then the following conditions are equivalent:
    \begin{enumerate}
        \item $(A,\star,[-,-])$ is a TALO algebra.
        \item The triple $(A,\cdot,[-,-])$  is a transposed Poisson algebra with a representation $(-\mathcal{L}_{\star},\mathrm{ad}$,
        $A)$, where $\cdot$ is defined by Eq.~(\ref{comm op}).
        \item There is a transposed Poisson algebra structure on $A\oplus A$ in which the commutative associative product $\cdot$ is defined by Eq.~(\ref{Zinbiel axiom}) and the Lie bracket $[-,-]$ is defined by Eq.~(\ref{Lie axiom}).
    \end{enumerate}
\end{cor}

\subsection{TAPO algebras}

\begin{defi}
    A \textbf{TAPO algebra} is a triple $(A,\star,\circ)$, such that $(A,\star)$ is an anti-Zinbiel algebra, $(A,\circ)$ is a pre-Lie algebra, and Eq.~(\ref{PTPA2}) and the following equations hold:
    \begin{equation}\label{TAPO1}
        2x\star(y\circ z)=-(x\star y+y\star x)\circ z+y\circ(x\star z),
    \end{equation}
    \begin{equation}\label{TAPO3}
    (z\star x+x\star z)\circ y-(z\star y+y\star z)\circ x=0,
\end{equation}
for all $x,y,z\in A.$
\end{defi}

\begin{pro}
    Let $(A,\cdot,[-,-])$ be a transposed Poisson algebra, $(A,\star)$ be a compatible anti-Zinbiel algebra of $(A,\cdot)$ and $(A,\circ)$ be a compatible pre-Lie algebra of $(A,[-,-])$. If  $(-\mathcal{L}_{\star},\mathcal{L}_{\circ},A)$ is a representation of $(A,\cdot,[-,-])$, then $(A,\star,\circ)$ is a TAPO algebra. Conversely, let $(A,\star,\circ)$ be a TAPO algebra, $(A,\cdot)$ be the sub-adjacent commutative associative algebra and $(A,[-,-])$ be the sub-adjacent Lie algebra of $(A,\circ)$. Then $(A,\cdot,[-,-])$ is a transposed Poisson algebra with a representation $(-\mathcal{L}_{\star},\mathcal{L}_{\circ},A)$. In this case, we say $(A,\cdot,[-,-])$ is the \textbf{sub-adjacent transposed Poisson algebra} of $(A,\star,\circ)$, and $(A,\star,\circ)$ is a \textbf{compatible   TAPO algebra} of $(A,\cdot,[-,-])$.
\end{pro}
\begin{proof}
    Since $(-\mathcal{L}_{\star},\mathcal{L}_{\circ},A)$ is a representation of $(A,\cdot,[-,-])$, Eqs.~(\ref{PTPA2}) and (\ref{TAPO1}) hold.
    Thus for all $x,y,z\in A$, we have
    \begin{eqnarray*}
        0&&=2[x,y]\cdot z-[z\cdot x,y]-[x,z\cdot y]\\
        &&\overset{(\ref{TAPO1}),(\ref{PTPA2})}{=}x\circ(y\star z)-y\circ(x\star z)-(z\cdot x)\circ y+x\circ(z\star y)+(z\cdot y)\circ x\\
        &&\ \ \ \ -y\circ(z\star x)-(z\cdot x)\circ y+y\circ(z\cdot x)-x\circ(z\cdot y)+(z\cdot y)\circ x\\
        &&=2(z\cdot y)\circ x-2(z\cdot x)\circ y.
    \end{eqnarray*}
Hence Eq.~(\ref{TAPO3}) holds, and thus $(A,\star,\circ)$ is a TAPO algebra. The converse  part is proved similarly.
\end{proof}

Hence  we get the following conclusion.

\begin{cor}
    Let $A$ be a vector space with  two bilinear operations $\star,\circ:A\otimes A\rightarrow A$. Then the following conditions are equivalent:
    \begin{enumerate}
        \item $(A,\star,\circ)$ is a TAPO algebra.
        \item The triple $(A,\cdot,[-,-])$   is a transposed Poisson algebra with a representation $(-\mathcal{L}_{\star},\mathcal{L}_{\circ}$,
        $A)$, where $\cdot$ and $[-,-]$ are respectively defined by Eqs.~(\ref{comm op}) and (\ref{sub-adj}).
        \item There is a transposed Poisson algebra structure on $A\oplus A$
        in which the commutative associative product $\cdot$ is defined by
        Eq.~(\ref{anti-Zinbiel axiom}) and the Lie bracket $[-,-]$ is defined by Eq.~(\ref{pre-Lie axiom}).
    \end{enumerate}
\end{cor}

\subsection{TAAO algebras}

\begin{defi}
    A \textbf{TAAO algebra} is a triple $(A,\star,\circ)$, such that $(A,\star)$ is an anti-Zinbiel algebra, $(A,\circ)$ is an anti-pre-Lie algebra, and Eqs.~(\ref{TZAO2}),(\ref{TAPO1}) and the following equation hold:
        \begin{equation}\label{TAAO3}
        (z\star x+x\star z)\circ y-(z\star y+y\star z)\circ x+x\circ(y\star z)-y\circ(x\star z)=0,\;\;\forall x,y,z\in A.
    \end{equation}
\end{defi}

\begin{pro}
    Let $(A,\cdot,[-,-])$ be a transposed Poisson algebra, $(A,\star)$ be a compatible anti-Zinbiel algebra of $(A,\cdot)$ and $(A,\circ)$ be a compatible anti-pre-Lie algebra of $(A,[-,-])$. If $(-\mathcal{L}_{\star},-\mathcal{L}_{\circ},A)$ is a representation of $(A,\cdot,[-,-])$,
     then $(A,\star,\circ)$ is a TAAO algebra.
    Conversely, let $(A,\star,\circ)$ be a TAAO algebra, $(A,\cdot)$ be the sub-adjacent commutative associative algebra of $(A,\star)$ and $(A,[-,-])$ be the sub-adjacent Lie algebra of $(A,\circ)$. Then $(A,\cdot,[-,-])$ is a transposed Poisson algebra with a representation $(-\mathcal{L}_{\star},-\mathcal{L}_{\circ},A)$. In this case, we say $(A,\cdot,[-,-])$ is the \textbf{sub-adjacent transposed Poisson algebra} of $(A,\star,\circ)$, and $(A,\star,\circ)$ is a \textbf{compatible TAAO algebra} of $(A,\cdot,[-,-])$.
\end{pro}
\begin{proof}
    Since $(-\mathcal{L}_{\star},-\mathcal{L}_{\circ},A)$ is a representation of $(A,\cdot,[-,-])$, Eqs.~(\ref{TZAO2}) and (\ref{TAPO1}) hold.
   Thus for all $x,y,z\in A$, we have
    \begin{eqnarray*}
        0&&=2[x,y]\cdot z-[z\cdot x,y]-[x,z\cdot y]\\
        &&\overset{(\ref{TZAO2}),(\ref{TAPO1})}{=}y\circ(x\star z)-x\circ(y\star z)-(z\cdot x)\circ y+x\circ(z\star y)+(z\cdot y)\circ x\\
        &&\ \ \ \ -y\circ(z\star x)-(z\cdot x)\circ y+y\circ(z\cdot x)-x\circ(z\cdot y)+(z\cdot y)\circ x\\
        &&=-2(z\cdot x)\circ y+2(z\cdot y)\circ x-2x\circ(y\star z)+2y\circ(x\star z).
    \end{eqnarray*}
Hence Eq.~(\ref{TAAO3}) holds, and thus $(A,\star,\circ)$ is a TAAO algebra. The converse part is proved similarly.
\end{proof}

Hence  we get the following conclusion.

\begin{cor}
    Let $A$ be a vector space with  two bilinear operations $\star,\circ:A\otimes A\rightarrow A$. Then the following conditions are equivalent:
    \begin{enumerate}
        \item $(A,\star,\circ)$ is a   TAAO algebra.
        \item The triple $(A,\cdot,[-,-])$ is a transposed Poisson algebra with a representation $(-\mathcal{L}_{\star},-\mathcal{L}_{\circ}$,
        $A)$, where $\cdot$ and $[-,-]$ are respectively defined by Eqs.~(\ref{comm op}) and (\ref{sub-adj}).
        \item There is a transposed Poisson algebra structure on $A\oplus A$ in which the commutative associative product $\cdot$ is defined by Eq.~(\ref{anti-Zinbiel axiom}) and the Lie bracket $[-,-]$ is defined by Eq.~(\ref{anti-pre-Lie axiom}).
    \end{enumerate}
\end{cor}


\subsection{Summary}\

We summarize some facts on the 8 algebraic structures in the previous subsections respectively corresponding to the mixed splittings of operations of transposed Poisson algebras in terms of representations of transposed
Poisson algebras on the  spaces themselves in Table 2.

\begin{small}
    \begin{table}[!t]
        \caption{Splittings of transposed Poisson algebras on the spaces themselves}
        %
        %
        \begin{tabular}{p{2cm}p{2.5cm}p{7cm}}
            \hline
            Algebras & Notations & Representations of transposed
            Poisson

            algebras on  the spaces themselves  \\
            TCPO   & $(A,\cdot,\circ)$  & $(\mathcal{L}_{\cdot},\mathcal{L}_{\circ},A)$ \\

            TCAO  & $(A,\cdot,\circ)$ & $(\mathcal{L}_{\cdot},-\mathcal{L}_{\circ},A)$  \\

            TZLO & $(A,\star,[-,-])$ & $(\mathcal{L}_{\star},\mathrm{ad},A)$ \\

            TZPO & $(A,\star,\circ)$ & $(\mathcal{L}_{\star},\mathcal{L}_{\circ},A)$  \\

            TZAO & $(A,\star,\circ)$ & $(\mathcal{L}_{\star},-\mathcal{L}_{\circ},A)$ \\

            TALO & $(A,\star,[-,-])$ & $(-\mathcal{L}_{\star},\mathrm{ad},A)$ \\

            TAPO & $(A,\star,\circ)$ & $(-\mathcal{L}_{\star},\mathcal{L}_{\circ},A)$  \\

            TAAO & $(A,\star,\circ)$ & $(-\mathcal{L}_{\star},-\mathcal{L}_{\circ},A)$  \\
            \hline
        \end{tabular}
    \end{table}
\end{small}

\section{Mixed splittings of operations of transposed Poisson algebras in terms of representations on the dual spaces
and related algebraic structures}\label{sec:4}\

We introduce 8 algebraic structures respectively corresponding to the mixed splitting of  the commutative associative
 products and Lie brackets of transposed Poisson algebras interlacedly in three manners: the classical splitting, the second splitting and the un-splitting, in terms of representations of transposed Poisson algebras on the dual spaces.
The relationships between
transposed Poisson algebras with nondegenerate bilinear forms satisfying certain conditions and  some algebraic structures are given.

Recall that for a Lie algebra $(\mathfrak{g},[-,-])$, a pair $(\rho,V)$ is a representation if and only if $(\rho^{*},V^{*})$ is a representation. Hence by Propositions \ref{pro:pre-Lie} and \ref{pro:anti-pre-Lie}, we have the following equivalent characterizations of pre-Lie algebras and anti-pre-Lie algebras in terms of the representations of Lie algebras on the dual spaces respectively.

\begin{pro}\label{pro:pre-Lie2}
    Let $A$ be a vector space together with a bilinear operation $\circ:A\otimes A\rightarrow A$. Then the following conditions are equivalent:
    \begin{enumerate}
        \item $(A,\circ)$ is a pre-Lie algebra.
        \item $(A,\circ)$ is a Lie-admissible algebra such that $(\mathcal{L}^{*}_{\circ},A^{*})$ is a representation of the sub-adjacent Lie algebra $(A,[-,-])$.
        \item There is a Lie algebra structure on $A\oplus A^{*}$ defined by
        \begin{equation}\label{pre-Lie axiom2}
            [(x,a^{*}),(y,b^{*})]=(x\circ y-y\circ x, \mathcal{L}^{*}_{\circ}(x)b^{*}-\mathcal{L}^{*}_{\circ}(y)a^{*}),\;\;\forall x,y\in A,a^{*},b^{*}\in A^{*}.
        \end{equation}
    \end{enumerate}
\end{pro}

\begin{pro}\label{pro:anti-pre-Lie2}
    Let $A$ be a vector space together with a bilinear operation $\circ:A\otimes A\rightarrow A$. Then the following conditions are equivalent:
    \begin{enumerate}
        \item $(A,\circ)$ is an anti-pre-Lie algebra.
        \item $(A,\circ)$ is a Lie-admissible algebra such that $(-\mathcal{L}^{*}_{\circ},A^{*})$ is a representation of the sub-adjacent Lie algebra $(A,[-,-])$.
        \item There is a Lie algebra structure on $A\oplus A^{*}$ defined by
        \begin{equation}\label{anti-pre-Lie axiom2}
            [(x,a^{*}),(y,b^{*})]=(x\circ y-y\circ x, \mathcal{L}^{*}_{\circ}(y)a^{*}-\mathcal{L}^{*}_{\circ}(x)b^{*}),\;\;\forall x,y\in A,a^{*},b^{*}\in A^{*}.
        \end{equation}
    \end{enumerate}
\end{pro}

Similarly, for a commutative associative algebra $(A,\cdot)$, $(\mu,V)$ is a representation if and only if $(-\mu^{*},V^{*})$ is a representation. Hence by Propositions \ref{pro:Zinbiel} and \ref{pro:anti-Zinbiel}, we have the following equivalent characterization of Zinbiel algebras and anti-Zinbiel algebras in terms of the representations of commutative associative algebras on the dual spaces respectively.

\begin{pro}\label{pro:Zinbiel2}
    Let $A$ be a vector space together with a bilinear operation $\star:A\otimes A\rightarrow A$. Then the following conditions are equivalent:
    \begin{enumerate}
        \item $(A,\star)$ is a Zinbiel algebra.
        \item $(A,\cdot)$ with the bilinear operation defined by Eq.~(\ref{comm op}) is a commutative associative algebra, and
        $(-\mathcal{L}^{*}_{\star},A^{*})$ is a representation of $(A,\cdot)$.
        \item There is a commutative associative algebra structure on $A\oplus A^{*}$ defined by
        \begin{equation}\label{Zinbiel axiom2}
            (x,a^{*})\cdot(y,b^{*})=(x\star y+y\star x,-\mathcal{L}^{*}_{\star}(x)b^{*}-\mathcal{L}^{*}_{\star}(y)a^{*}),\;\;\forall x,y\in A,a^{*},b^{*}\in A^{*}.
        \end{equation}
    \end{enumerate}
\end{pro}

\begin{pro}\label{pro:anti-Zinbiel2}
    Let $A$ be a vector space together with a bilinear operation $\star:A\otimes A\rightarrow A$. Then the following conditions are equivalent:
    \begin{enumerate}
        \item $(A,\star)$ is an anti-Zinbiel algebra.
        \item $(A,\cdot)$ with the bilinear operation defined by Eq.~(\ref{comm op}) is a commutative associative algebra, and
        $(\mathcal{L}^{*}_{\star},A^{*})$ is a representation of $(A,\cdot)$.
        \item There is a commutative associative algebra structure on $A\oplus A^{*}$ defined by
        \begin{equation}\label{anti-Zinbiel axiom2}
            (x,a^{*})\cdot(y,b^{*})=(x\star y+y\star x,\mathcal{L}^{*}_{\star}(x)b^{*}+\mathcal{L}^{*}_{\star}(y)a^{*}),\;\;\forall x,y\in A,a^{*},b^{*}\in A^{*}.
        \end{equation}
    \end{enumerate}
\end{pro}

For a representation $(\mu,\rho,V)$ of a transposed Poisson algebra $(A,\cdot,[-,-])$, $(-\mu^{*}, \rho^{*},V^{*})$ is not necessarily a representation of $(A,\cdot,[-,-])$. In fact, we have

\begin{pro}\label{pro:TPA dual rep}
    Let $(A,\cdot,[-,-])$ be a transposed Poisson algebra, $(\mu,V)$ be a representation of $(A,\cdot)$ and $(\rho,V)$ be a representation of $(A,[-,-])$. Then $(-\mu^{*}, \rho^{*},V^{*})$ is a representation of $(A,\cdot,[-,-])$ if and only if
        \begin{equation}\label{TPA dual rep1}
        2\rho(y)\mu(x)=\rho(x\cdot y)+\mu(x)\rho(y),
    \end{equation}
    \begin{equation}\label{TPA dual rep2}
        2\mu([x,y])=\mu(x)\rho(y)-\mu(y)\rho(x),
    \end{equation}
for all $x,y\in A$.
\end{pro}
\begin{proof}
    Let $x,y\in A, u^{*}\in V^{*}, v\in V$. Then we have
    \begin{eqnarray*}
        &&\langle(\rho^{*}(x\cdot y)-\rho^{*}(y)\mu^{*}(x)+2\mu^{*}(x)\rho^{*}(y))u^{*}, v\rangle\\
        &&=\langle u^{*}, (-\rho(x\cdot y)-\mu(x)\rho(y)+2\rho(y)\mu(x))v\rangle,\\
        &&\langle (-\rho^{*}(x)\mu^{*}(y)+\rho^{*}(y)\mu^{*}(x)+2\mu^{*}([x,y]))u^{*},v\rangle\\
        &&=\langle u^{*},(-\mu(y)\rho(x)+\mu(x)\rho(y)-2\mu([x,y]))v\rangle.
    \end{eqnarray*}
    Hence the conclusion follows.
\end{proof}

\begin{pro}\label{pro:dual rep TPA}
    Let $(A,\cdot,[-,-])$ be a transposed Poisson algebra and
    $(\mu,\rho,V)$ be a representation of $(A,\cdot,[-,-])$. Then
    $(-\mu^{*},\rho^{*},V^{*})$ is a representation of
    $(A,\cdot,[-,-])$ if and only if
    \begin{equation}\label{eq:d0}
        \mu([x,y])=0,\ \rho(x\cdot y)=\mu(x)\rho(y),\;\;\forall x,y\in A.
    \end{equation}
    In particular, $(-\mathcal{L}^{*}_{\cdot},\mathrm{ad}^{*},A^{*})$
    is a representation of $(A,\cdot,[-,-])$ if and only if Eq.~(\ref{trivial TPA}) holds.
\end{pro}

\begin{proof}
    By the assumption that  Eqs.~(\ref{eq:defi:TPA REP1}) and (\ref{eq:defi:TPA REP2}) hold,
     it is straightforward to show that
    Eqs.~(\ref{TPA dual rep1}) and (\ref{TPA dual rep2}) hold if and only if
    Eq.~(\ref{eq:d0}) holds. Hence the conclusion follows from Proposition \ref{pro:TPA dual rep}.
\end{proof}


Hence  we get the following conclusion.

\begin{cor}
Let $A$ be a vector space with two bilinear operations $\cdot,[-,-]:A\otimes A\rightarrow A$. Then the following conditions are equivalent:
\begin{enumerate}
    \item $(A,\cdot)$ is a commutative associative algebra, $(A,[-,-])$ is a Lie algebra, and Eq.~(\ref{trivial TPA}) holds.
    \item $(A,\cdot,[-,-])$ is a  transposed Poisson algebra with a representation $(-\mathcal{L}^{*}_{\cdot},\mathrm{ad}^{*},A^{*})$.
    \item There is a transposed Poisson algebra structure on $A\oplus A^{*}$ in which the commutative associative product is defined by
    \begin{equation}\label{comm asso axiom2}
        (x,a^{*})\cdot(y,b^{*})=(x\cdot y,-\mathcal{L}^{*}_{\cdot}(x)b^{*}-\mathcal{L}^{*}_{\cdot}(y)a^{*}),
    \end{equation}
and the Lie bracket is defined by
\begin{equation}\label{Lie axiom2}
    [(x,a^{*}),(y,b^{*})]=([x,y],\mathrm{ad}^{*}(x)b^{*}-\mathrm{ad}^{*}(y)a^{*}),
\end{equation}
for all $x,y\in A, a^{*},b^{*}\in A^{*}$.
\end{enumerate}
    \end{cor}

\begin{pro}
    Let $(A.\cdot,[-,-])$ be a transposed Poisson algebra. Suppose that there is a nondegenerate symmetric bilinear from $\mathcal{B}$ on $A$ such that it is invariant on both $(A,\cdot)$ and $(A,[-,-])$. Then
    Eq.~(\ref{trivial TPA}) holds. Conversely, suppose that $(A,\cdot,[-,-])$ is a transposed Poisson algebra and Eq.~(\ref{trivial TPA}) holds. Then there is a transposed Poisson algebra $A\ltimes_{-\mathcal{L}^{*}_{\cdot},\mathrm{ad}^{*}}A^{*}$, and the natural nondegenerate symmetric bilinear form $\mathcal{B}_{d}$ defined by Eq.~(\ref{B_{d}}) is  invariant on both the commutative associative algebra $A\ltimes_{-\mathcal{L}^{*}_{\cdot}}A^{*}$ and the Lie algebra $A\ltimes_{\mathrm{ad}^{*}}A^{*}$.
    \end{pro}
\begin{proof}
    It follows from a direct checking.
    \end{proof}

Next we introduce 8 algebraic structures in the rest of this section  corresponding to the mixed splitting of  the commutative associative
 products and the Lie brackets of transposed Poisson algebras interlacedly in three manners: the classical splitting, the second splitting and the un-splitting, in terms of representations of transposed Poisson algebras on the dual spaces.
 We still use the principle given in the previous section to name them.

\subsection{TCPD algebras}

\begin{defi}
    A \textbf{TCPD algebra} is a triple $(A,\cdot,\circ)$, such that $(A,\cdot)$ is a commutative associative algebra, $(A,\circ)$ is a pre-Lie algebra, and the following equations hold:
    \begin{equation}\label{TCPD1}
        2x\circ(y\cdot z)=(z\cdot x)\circ y+z\cdot(x\circ y),
    \end{equation}
    \begin{equation}\label{TCPD2}
        2(x\circ y)\cdot z-2(y\circ x)\cdot z=x\cdot (y\circ z)-y\cdot(x\circ z),
    \end{equation}
        \begin{equation}\label{TCPD3}
        3y\circ(z\cdot x)-3x\circ(z\cdot y)-(z\cdot x)\circ y+(z\cdot y)\circ x=0,
    \end{equation}
for all $x,y,z\in A$.
\end{defi}

\begin{pro}
    Let $(A,\cdot,[-,-])$ be a transposed Poisson algebra  and $(A,\circ)$ be a compatible pre-Lie algebra of $(A,[-,-])$. If $(-\mathcal{L}^{*}_{\cdot},\mathcal{L}^{*}_{\circ}, A^{*})$ is a representation of $(A,\cdot,[-,-])$, then $(A,\cdot,\circ)$ is a TCPD algebra. Conversely, let $(A,\cdot,\circ)$ be a TCPD algebra  and  $(A,[-,-])$ be the sub-adjacent Lie algebra of $(A,\circ)$. Then $(A,\cdot,[-,-])$ is a transposed Poisson algebra with a representation $(-\mathcal{L}^{*}_{\cdot},\mathcal{L}^{*}_{\circ}, A^{*})$.
    In this case, we say $(A,\cdot,[-,-])$ is the \textbf{sub-adjacent transposed Poisson algebra} of $(A,\cdot,\circ)$, and $(A,\cdot,\circ)$ is a \textbf{compatible  TCPD algebra} of $(A,\cdot,[-,-])$.
    \end{pro}
\begin{proof}
    Since $(-\mathcal{L}^{*}_{\cdot},\mathcal{L}^{*}_{\circ}, A^{*})$ is a representation of $(A,\cdot,[-,-])$, we get Eqs.~(\ref{TCPD1})-(\ref{TCPD2}).
      Thus   for all $x,y,z\in A$, we have
        \begin{eqnarray*}
            0&&=2z\cdot[x,y]-[z\cdot x,y]-[x,z\cdot y]\\
            &&\overset{(\ref{TCPD2})}{=}x\cdot(y\circ z)-y\cdot(x\circ z)-(z\cdot x)\circ y+y\circ(z\cdot x)-x\circ(z\cdot y)+(z\cdot y)\circ x\\
            &&\overset{(\ref{TCPD1})}{=}3y\circ(z\cdot x)-3x\circ(z\cdot y)-(z\cdot x)\circ y+(z\cdot y)\circ x.
        \end{eqnarray*}
    Hence Eq.~(\ref{TCPD3}) holds, and thus $(A,\cdot,\circ)$ is a TCPD algebra. The converse part is proved similarly.
    \end{proof}

Hence  we get the following conclusion.

\begin{cor}
    Let $A$ be a vector space with two bilinear operations $\cdot,\circ:A\otimes A\rightarrow A$. Then the following conditions are equivalent:
    \begin{enumerate}
        \item $(A,\cdot,\circ)$ is a TCPD algebra.
        \item The triple $(A,\cdot,[-,-])$   is a transposed Poisson algebra with a representation $(-\mathcal{L}^{*}_{\cdot},\mathcal{L}^{*}_{\circ}$,
        $A^{*})$, where $[-,-]$ is defined by Eq.~(\ref{sub-adj}).
        \item There is a transposed Poisson algebra structure on $A\oplus A^{*}$ in which the commutative associative product $\cdot$ is defined by Eq.~(\ref{comm asso axiom2}) and the Lie bracket $[-,-]$ is defined by Eq.~(\ref{pre-Lie axiom2}).
    \end{enumerate}
    \end{cor}

\subsection{Anti-pre-Lie-Poisson algebras or TCAD algebras}

\begin{defi}\cite{LB2022}
    An \textbf{anti-pre-Lie Poisson algebra} or a {\bf TCAD algebra} is a triple $(A,\cdot,\circ)$, such that $(A,\cdot)$ is a commutative associative algebra, $(A,\circ)$ is an anti-pre-Lie algebra, and Eq.~(\ref{TCPD1}) and the following equation hold:
    \begin{equation}
        2(x\circ y)\cdot z-2(y\circ x)\cdot z=y\cdot(x\circ z)-x\cdot (y\circ z),\;\;\forall x,y,z\in A.
    \end{equation}
\end{defi}

\begin{pro}\cite{LB2022}\
    Let $(A,\cdot,[-,-])$ be a transposed Poisson algebra and  $(A,\circ)$ be a compatible anti-pre-Lie algebra of $(A,[-,-])$. If $(-\mathcal{L}^{*}_{\cdot},-\mathcal{L}^{*}_{\circ},A^{*})$  is a representation of $(A,\cdot,[-,-])$, then  $(A,\cdot,\circ)$ is an anti-pre-Lie Poisson algebra. Conversely,
    let $(A,\cdot,\circ)$ be an anti-pre-Lie Poisson algebra and $(A,[-,-])$ be the sub-adjacent Lie algebra of $(A,\circ)$. Then $(A,\cdot,[-,-])$ is a transposed Poisson algebra
    with a representation $(-\mathcal{L}^{*}_{\cdot},-\mathcal{L}^{*}_{\circ},A^{*})$.
        In this case, we say $(A,\cdot,[-,-])$ is the \textbf{sub-adjacent transposed Poisson algebra} of $(A,\cdot,\circ)$, and $(A,\cdot,\circ)$ is a \textbf{compatible anti-pre-Lie Poisson algebra} of $(A,\cdot,[-,-])$.
\end{pro}


\begin{ex}\cite{LB2022}
    Let $(A,\cdot)$ be a commutative associative algebra with a derivation $P$.
    Then there is an anti-pre-Lie algebra $(A,\circ)$ defined by
    \begin{equation}
        x\circ y=P(x\cdot y)+P(x)\cdot y,\;\;\forall x,y\in A.
    \end{equation}
    Moreover, $(A,\cdot,\circ)$ is an anti-pre-Lie Poisson algebra and for the sub-adjacent transposed Poisson algebra $(A,\cdot,[-,-])$, the following equation holds:
    \begin{equation}
        [x,y]=P(x)\cdot y-x\cdot P(y),\;\;\forall x,y\in A.
    \end{equation}
\end{ex}

Moreover  we get the following conclusion.

\begin{cor}
    Let $A$ be a vector space with two bilinear operations $\cdot,\circ:A\otimes A\rightarrow A$. Then the following conditions are equivalent:
    \begin{enumerate}
        \item $(A,\cdot,\circ)$ is an anti-pre-Lie Poisson algebra.
        \item The triple $(A,\cdot,[-,-])$ is a transposed Poisson algebra with a representation $(-\mathcal{L}^{*}_{\cdot}$,
        $-\mathcal{L}^{*}_{\circ},A^{*})$, where $[-,-]$ is defined by Eq.~(\ref{sub-adj}).
        \item There is a transposed Poisson algebra structure on $A\oplus A^{*}$, in which the commutative associative product $\cdot$ is defined by Eq.~(\ref{comm asso axiom2}) and the Lie bracket $[-,-]$ is defined by  Eq.~(\ref{anti-pre-Lie axiom2}).
    \end{enumerate}
\end{cor}

\begin{pro}\cite{LB2022}\label{pro:anti-pre-Lie Poisson}\
    Let $(A.\cdot,[-,-])$ be a transposed Poisson algebra. Suppose that there is a nondegenerate symmetric bilinear from $\mathcal{B}$ on $A$ such that it is invariant on   $(A,\cdot)$ and a commutative 2-cocycle on $(A,[-,-])$. Then there is a compatible anti-pre-Lie Poisson algebra $(A,\cdot,\circ)$ in which $\circ$ is defined by Eq.~(\ref{bilinear form anti-pre-Lie}). Conversely, let $(A,\cdot,\circ)$ be an anti-pre-Lie Poisson algebra and the sub-adjacent transposed Poisson algebra be $(A,\cdot,[-,-])$. Then there is a transposed Poisson algebra $A\ltimes_{-\mathcal{L}^{*}_{\cdot},-\mathcal{L}^{*}_{\circ}}A^{*}$, and the natural nondegenerate symmetric bilinear form $\mathcal{B}_{d}$ defined by Eq.~(\ref{B_{d}}) is  invariant on  the commutative associative algebra $A\ltimes_{-\mathcal{L}^{*}_{\cdot}}A^{*}$ and a commutative 2-cocycle on the Lie algebra $A\ltimes_{-\mathcal{L}^{*}_{\circ}}A^{*}$.
\end{pro}

\subsection{TZLD algebras}

\begin{defi}
    A \textbf{TZLD algebra} is a triple $(A,\star,[-,-])$, such that $(A,\star)$ is a Zinbiel algebra, $(A,[-,-])$ is a Lie algebra, and the following equations hold:
    \begin{equation}\label{TZLD1}
        2[y,x\star z]=[x\star y+y\star x,z]+x\star[y,z],
    \end{equation}
    \begin{equation}\label{TZLD2}
        2[x,y]\star z=x\star[y,z]-y\star[x,z],
    \end{equation}
        \begin{equation}\label{TZLD3}
       2z\star[x,y]+[y,x\star z+z\star x]-[x,y\star z+z\star y]=0,
    \end{equation}
for all $x,y,z\in A$.
\end{defi}

\begin{pro}
    Let $(A,\cdot,[-,-])$ be a transposed Poisson algebra and $(A,\star)$ be a compatible Zinbiel algebra of $(A,\cdot)$. If $(-\mathcal{L}^{*}_{\star},\mathrm{ad}^{*},A^{*})$ is a representation of $(A,\cdot,[-,-])$, then $(A,\star,[-,-])$ is a TZLD algebra. Conversely, let $(A,\star,[-,-])$ be a TZLD algebra and $(A,\cdot)$ be the sub-adjacent commutative associative algebra of $(A,\star)$. Then $(A,\cdot,[-,-])$ is a transposed Poisson algebra with a representation $(-\mathcal{L}^{*}_{\star},\mathrm{ad}^{*},A^{*})$. In this case, we say $(A,\cdot,[-,-])$ is the \textbf{sub-adjacent transposed Poisson algebra} of $(A,\star,[-,-])$, and $(A,\star,[-,-])$ is a \textbf{compatible TZLD algebra} of $(A,\cdot,[-,-])$.
\end{pro}
\begin{proof}
   Since $(-\mathcal{L}^{*}_{\star},\mathrm{ad}^{*},A^{*})$ is a representation of $(A,\cdot,[-,-])$, we get Eqs.~(\ref{TZLD1}) and (\ref{TZLD2}).
    By Eq.~(\ref{TZLD1}), we have
    \begin{equation}\label{TZLD4}
        x\star[y,z]-y\star[x,z]=2[y,x\star z]-2[x,y\star z],\;\;\forall x,y,z\in A.
    \end{equation}
   Thus  for all $x,y,z\in A$, we have
    \begin{eqnarray*}
        0&&=2z\cdot[x,y]-[z\cdot x,y]-[x,z\cdot y]\\
        &&\overset{(\ref{TZLD1}),(\ref{TZLD2})}{=}x\star[y,z]-y\star[x,z]+2z\star[x,y]\\
        &&\ \ \ \ +z\star[x,y]-2[x,z\star y]-z\star[y,x]+2[y,z\star x]\\
        &&=x\star[y,z]-y\star[x,z]+4z\star[x,y]-2[x,z\star y]+2[y,z\star x]\\
        &&\overset{(\ref{TZLD4})}{=} 4z\star[x,y]+2[y,x\star z+z\star x]-2[x,y\star z+z\star y].
    \end{eqnarray*}
Hence Eq.~(\ref{TZLD3}) holds, and thus $(A,\star,[-,-])$ is a TZLD algebra. The converse part is proved similarly.
\end{proof}

Hence  we get the following conclusion.

\begin{cor}
    Let $A$ be a vector space with two bilinear operations $\cdot,[-,-]:A\otimes A\rightarrow A$. Then the following conditions are equivalent:
    \begin{enumerate}
        \item $(A,\star,[-,-])$ is a  TZLD algebra.
        \item The triple $(A,\cdot,[-,-])$   is a  transposed Poisson algebra with a representation $(-\mathcal{L}^{*}_{\star},\mathrm{ad}^{*}$,
        $A^{*})$, where $\cdot$ is defined by Eq.~(\ref{comm op}).
        \item There is a transposed Poisson algebra structure on $A\oplus A^{*}$ in which the commutative associative product $\cdot$ is defined  by Eq.~(\ref{Zinbiel axiom2}) and the Lie bracket $[-,-]$ is defined by Eq.~(\ref{Lie axiom2}).
    \end{enumerate}
\end{cor}

\subsection{TZPD algebras}

\begin{defi}
    A \textbf{TZPD algebra} is a triple $(A,\star,\circ)$, such that $(A,\star)$ is a Zinbiel algebra, $(A,\circ)$ is a pre-Lie algebra, and the following equations hold:
    \begin{equation}\label{TZPD1}
        2y\circ(x\star z)=(x\star y+y\star x)\circ z+x\star(y\circ z),
    \end{equation}
    \begin{equation}\label{TZPD2}
        2(x\circ y-y\circ x)\star z=x\star(y\circ z)-y\star(x\circ z),
    \end{equation}
    \begin{equation}\label{TZPD3}
    y\circ(x\star z+z\star x)-x\circ(y\star z+z\star y)+z\star(x\circ y-y\circ x)=0,
    \end{equation}
for all $x,y,z\in A$.
\end{defi}

\begin{pro}
    Let $(A,\cdot,[-,-])$ be a transposed Poisson algebra, $(A,\star)$ be a compatible Zinbiel algebra of $(A,\cdot)$  and $(A,\circ)$ be a compatible pre-Lie algebra of $(A,[-,-])$. If $(-\mathcal{L}^{*}_{\star},\mathcal{L}^{*}_{\circ},A^{*})$ is a representation of $(A,\cdot,[-,-])$, then   $(A,\star,\circ)$ is a TZPD algebra. Conversely, let $(A,\star,\circ)$ be a TZPD algebra, $(A,\cdot)$ be the sub-adjacent commutative associative algebra of $(A,\star)$  and $(A,[-,-])$ be the sub-adjacent Lie algebra of $(A,\circ)$. Then $(A,\cdot,[-,-])$ is a transposed Poisson algebra with  a representation $(-\mathcal{L}^{*}_{\star},\mathcal{L}^{*}_{\circ},A^{*})$. In this case, we say $(A,\cdot,[-,-])$ is the \textbf{sub-adjacent transposed Poisson algebra} of $(A,\cdot,\circ)$, and $(A,\cdot,\circ)$ is a \textbf{compatible   TZPD algebra} of $(A,\cdot,[-,-])$.
\end{pro}
\begin{proof}
   Since $(-\mathcal{L}^{*}_{\star},\mathcal{L}^{*}_{\circ},A^{*})$ is a representation of $(A,\cdot,[-,-])$, we get Eqs.~(\ref{TZPD1})-(\ref{TZPD2}).
    By Eq.~(\ref{TZPD1}), we have
    \begin{equation}\label{TZPD4}
        x\star(y\circ z)-y\star(x\circ z)=2y\circ(x\star z)-2x\circ(y\star z),\;\;\forall x,y,z\in A.
    \end{equation}
    Thus for all $x,y,z\in A$, we have
    \begin{eqnarray*}
        0&&=2z\cdot[x,y]-[z\cdot x,y]-[x,z\cdot y]\\
        &&=2[x,y]\star z+2z\star[x,y]-(z\cdot x)\circ y+y\circ(z\cdot x)-x\circ(z\cdot y)+(z\cdot y)\circ x\\
        &&\overset{(\ref{TZPD2})}{=}x\star(y\circ z)-y\star(x\circ z)+2z\star(x\circ y)-2z\star(y\circ x)\\
        &&\ \ \ \ -(z\star x)\circ y-(x\star z)\circ y+y\circ(z\star x)+y\circ(x\star z)\\
        &&\ \ \ \ -x\circ(z\star y)-x\circ(y\star z)+(y\star z)\circ x+(z\star y)\circ x\\
        &&\overset{(\ref{TZPD1}),(\ref{TZPD4})}{=}3y\circ(x\star z+z\star x)-3x\circ(y\star z+z\star y)+3z\star(x\circ y-y\circ x).
    \end{eqnarray*}
Hence Eq.~(\ref{TZPD3}) holds, and thus $(A,\star,\circ)$ is a TZPD algebra. The converse part is proved similarly.
\end{proof}

Hence  we get the following conclusion.

\begin{cor}
    Let $A$ be a vector space with two bilinear operations $\star,\circ:A\otimes A\rightarrow A$. Then the following conditions are equivalent:
    \begin{enumerate}
        \item $(A,\star,\circ)$ is a  TZPD algebra.
        \item The triple $(A,\cdot,[-,-])$   is a  transposed Poisson algebra with a representation $(-\mathcal{L}^{*}_{\star},\mathcal{L}^{*}_{\circ}$,
        $A^{*})$, where $\cdot$ and $[-,-]$ are respectively defined by Eqs.~(\ref{comm op}) and (\ref{sub-adj}).
        \item There is a transposed Poisson algebra structure on $A\oplus A^{*}$ in which the commutative associative product $\cdot$ is defined  by Eq.~(\ref{Zinbiel axiom2}) and the Lie bracket $[-,-]$ is defined by Eq.~(\ref{pre-Lie axiom2}).
    \end{enumerate}
    \end{cor}

\begin{pro}
    Let $(A,\cdot,[-,-])$ be a transposed  Poisson algebra. Suppose that $\mathcal{B}$ is a nondegenerate skew-symmetric bilinear form on $A$ such that it is a Connes cocycle on $(A,\cdot)$ and a symplectic form on $(A,[-,-])$. Then there is a compatible TZPD algebra $(A,\star,\circ)$ in which $\star$ and $\circ$ are respectively defined by Eqs.~(\ref{bilinear form Zinbiel}) and (\ref{bilinear form pre-Lie}). Conversely, let $(A,\star,\circ)$ be a  TZPD algebra and the sub-adjacent transposed Poisson algebra be $(A,\cdot,[-,-])$. Then there is a transposed Poisson algebra $A\ltimes_{-\mathcal{L}^{*}_{\star},\mathcal{L}^{*}_{\circ}}A^{*}$, and the natural nondegenerate skew-symmetric bilinear form $\mathcal{B}_{p}$ defined by Eq.~(\ref{B_{p}})
    is a Connes cocycle on the commutative associative algebra $A\ltimes_{-\mathcal{L}^{*}_{\star}}A^{*}$ and a symplectic form on the Lie algebra $A\ltimes_{\mathcal{L}^{*}_{\circ}}A^{*}$.
\end{pro}
\begin{proof}
    It is similar to the proof of Proposition \ref{pro:anti-pre-Lie Poisson} given in \cite{LB2022}.
\end{proof}

\subsection{TZAD algebras}

\begin{defi}
    A \textbf{TZAD algebra} is a triple $(A,\star,\circ)$, such that $(A,\star)$ is a Zinbiel algebra, $(A,\circ)$ is an anti-pre-Lie algebra, and Eq.~(\ref{TZPD1}) and the following equations  hold:
    \begin{equation}\label{TZAD2}
        2(x\circ y-y\circ x)\star z=y\star(x\circ z)-x\star(y\circ z),
    \end{equation}
\begin{equation}\label{TZAD3}
    x\circ(y\star z)-y\circ(x\star z)-3x\circ(z\star y)+3y\circ(z\star x)+3z\star(x\circ y-y\circ x)=0,
\end{equation}
for all $x,y,z\in A$.
\end{defi}

\begin{pro}
    Let $(A,\cdot,[-,-])$ be a transposed Poisson algebra, $(A,\star)$ be a compatible Zinbiel algebra of $(A,\cdot)$ and $(A,\circ)$ be a compatible anti-pre-Lie algebra of $(A,[-,-])$. If $(-\mathcal{L}^{*}_{\star},-\mathcal{L}^{*}_{\circ},A^{*})$ is a representation of $(A,\cdot,[-,-])$, then   $(A,\star,\circ)$ is a TZAD algebra. Conversely, let $(A,\star,\circ)$ be a TZAD algebra, $(A,\cdot)$ be the sub-adjacent commutative associative algebra of $(A,\star)$ and $(A,[-,-])$ be the sub-adjacent Lie algebra of $(A,\circ)$. Then $(A,\cdot,[-,-])$ is a transposed Poisson algebra with a representation $(-\mathcal{L}^{*}_{\star},-\mathcal{L}^{*}_{\circ},A^{*})$.
In this case, we say $(A,\cdot,[-,-])$ is the \textbf{sub-adjacent transposed Poisson algebra} of $(A,\cdot,\circ)$, and $(A,\cdot,\circ)$ is a \textbf{compatible   TZAD algebra} of $(A,\cdot,[-,-])$.
\end{pro}
\begin{proof}
    Since $(-\mathcal{L}^{*}_{\star},-\mathcal{L}^{*}_{\circ},A^{*})$ is a representation of $(A,\cdot,[-,-])$, we get Eqs.~(\ref{TZPD1}) and (\ref{TZAD2}).
    Thus for all $x,y,z\in A$, we have
    \begin{eqnarray*}
        0&&=2z\cdot[x,y]-[z\cdot x,y]-[x,z\cdot y]\\
        &&=2[x,y]\star z+2z\star[x,y]-(z\cdot x)\circ y+y\circ(z\cdot x)-x\circ(z\cdot y)+(z\cdot y)\circ x\\
        &&\overset{(\ref{TZAD2})}{=}-x\star(y\circ z)+y\star(x\circ z)+2z\star(x\circ y)-2z\star(y\circ x)\\
        &&\ \ \ \ -(z\star x)\circ y-(x\star z)\circ y+y\circ(z\star x)+y\circ(x\star z)\\
        &&\ \ \ \ -x\circ(z\star y)-x\circ(y\star z)+(y\star z)\circ x+(z\star y)\circ x\\
        &&\overset{(\ref{TZPD1}),(\ref{TZPD4})}{=}x\circ(y\star z-3z\star y)-y\circ(x\star z-3z\star x)+3z\star(x\circ y-y\circ x).
    \end{eqnarray*}
Hence Eq.~(\ref{TZAD3}) holds, and thus $(A,\star,\circ)$ is a TZAD algebra. The converse part is proved similarly.
\end{proof}

Hence  we get the following conclusion.

\begin{cor}
    Let $A$ be a vector space with two bilinear operations $\cdot,[-,-]:A\otimes A\rightarrow A$. Then the following conditions are equivalent:
    \begin{enumerate}
        \item $(A,\star,\circ)$ is a   TZAD algebra.
        \item The triple $(A,\cdot,[-,-])$   is a  transposed Poisson algebra with a representation $(-\mathcal{L}^{*}_{\star}$,
        $-\mathcal{L}^{*}_{\circ},A^{*})$, where $\cdot$ and $[-,-]$ are respectively defined by Eqs.~(\ref{comm op}) and (\ref{sub-adj}).
        \item There is a transposed Poisson algebra structure on $A\oplus A^{*}$ in which the commutative associative product $\cdot$ is defined by  Eq.~(\ref{Zinbiel axiom2}) and the Lie bracket $[-,-]$ is defined by  Eq.~(\ref{anti-pre-Lie axiom2}).
    \end{enumerate}
\end{cor}

\subsection{TALD algebras}

\begin{defi}
    A \textbf{TALD algebra} is a triple $(A,\star,[-,-])$, such that $(A,\star)$ is an anti-Zinbiel algebra, $(A,[-,-])$ is a Lie algebra, and Eq.~(\ref{TZLD2}) and the following equations hold:
    \begin{equation}\label{TALD1}
        2[y,x\star z]=x\star[y,z]-[x\star y+y\star x,z],
    \end{equation}
\begin{equation}\label{TALD3}
    x\star[y,z]-y\star[x,z]+2[x,z\star y]-2[y,z\star x]=0,
\end{equation}
for all $x,y,z\in A$.
\end{defi}

\begin{pro}
    Let $(A,\cdot,[-,-])$ be a transposed Poisson algebra and $(A,\star)$ be a compatible anti-Zinbiel algebra of $(A,\cdot)$. If $(\mathcal{L}^{*}_{\star},\mathrm{ad}^{*},A^{*})$ is a representation of $(A,\cdot,[-,-])$, then $(A,\star,[-,-])$ is a TALD algebra. Conversely, let $(A,\star,[-,-])$ be a TALD algebra and $(A,\cdot)$ be the sub-adjacent commutative associative algebra of $(A,\star)$. Then $(A,\cdot,[-,-])$ is a transposed Poisson algebra with a representation $(\mathcal{L}^{*}_{\star},\mathrm{ad}^{*},A^{*})$. In this case, we say $(A,\cdot,[-,-])$ is the \textbf{sub-adjacent transposed Poisson algebra} of $(A,\star,[-,-])$, and $(A,\star,[-,-])$ is a \textbf{compatible TALD algebra} of $(A,\cdot,[-,-])$.
\end{pro}
\begin{proof}
    Since $(\mathcal{L}^{*}_{\star},\mathrm{ad}^{*},A^{*})$ is a representation of $(A,\cdot,[-,-])$, we get Eqs.~(\ref{TZLD2}) and (\ref{TALD1}).
        Thus for all $x,y,z\in A$, we have
    \begin{eqnarray*}
        &&2z\cdot[x,y]-[z\cdot x,y]-[x,z\cdot y]\\
        &&\overset{(\ref{TALD1}),(\ref{TZLD2})}{=}x\star[y,z]-y\star[x,z]+2z\star[x,y]-z\star[x,y]+2[x,z\star y]\\
        &&\ \ \ \ +z\star[y,x]-2[y,z\star x]\\
        &&=x\star[y,z]-y\star[x,z]+2[x,z\star y]-2[y,z\star x].
    \end{eqnarray*}
Hence Eq.~(\ref{TALD3}) holds, and thus $(A,\star,[-,-])$ is a TALD algebra. The converse part is proved similarly.
\end{proof}

Hence  we get the following conclusion.

\begin{cor}
    Let $A$ be a vector space  with two bilinear operations $\star,[-,-]:A\otimes A\rightarrow A$. Then the following conditions are equivalent:
    \begin{enumerate}
        \item $(A,\star,[-,-])$ is a   TALD algebra.
        \item The triple $(A,\cdot,[-,-])$   is a transposed Poisson algebra with a representation $(\mathcal{L}^{*}_{\star},\mathrm{ad}^{*}$,
        $A^{*})$, where $\cdot$ is defined by Eq.~(\ref{comm op}).
        \item There is a transposed Poisson algebra structure on $A\oplus A^{*}$ in which the commutative associative product $\cdot$ is defined  by
        Eq.~(\ref{anti-Zinbiel axiom2}) and the Lie bracket $[-,-]$ is defined by Eq.~(\ref{Lie axiom2}).
    \end{enumerate}
\end{cor}

\begin{pro}
    Let $(A,\cdot,[-,-])$ be a transposed Poisson algebra. Suppose that $\mathcal{B}$ is a nondegenerate symmetric bilinear form on $A$ such that it is a commutative Connes cocycle on $(A,\cdot)$ and invariant on $(A,[-,-])$. Then there is a compatible TALD algebra $(A,\star,[-,-])$ in which $\star$ is defined by Eq.~(\ref{bilinear form anti-Zinbiel}). Conversely, let $(A,\star,[-,-])$ be a  TALD algebra and the sub-adjacent transposed Poisson algebra be $(A,\cdot,[-,-])$. Then there is a transposed Poisson algebra $A\ltimes_{\mathcal{L}^{*}_{\star},\mathrm{ad}^{*}}A^{*}$, and the natural nondegenerate  symmetric bilinear form $\mathcal{B}_{d}$ defined by Eq.~(\ref{B_{d}})
    is a commutative Connes cocycle on the commutative associative algebra $A\ltimes_{\mathcal{L}^{*}_{\star}}A^{*}$ and invariant on the Lie algebra $A\ltimes_{\mathrm{ad}^{*}}A^{*}$.
\end{pro}
\begin{proof}
It is similar to the proof of Proposition \ref{pro:anti-pre-Lie Poisson} given in \cite{LB2022}.
\end{proof}

\subsection{TAPD algebras}

\begin{defi}
    A \textbf{TAPD algebra} is a triple $(A,\star,\circ)$, such that $(A,\star)$ is an anti-Zinbiel algebra, $(A,\circ)$ is a pre-Lie algebra, and Eq.~(\ref{TZPD2}) and the following equations hold:
    \begin{equation}\label{TAPD1}
        2y\circ(x\star z)=x\star(y\circ z)-(x\star y+y\star x)\circ z,
    \end{equation}
\begin{equation}\label{TAPD3}
    x\circ(z\star y)-y\circ(z\star x)+z\star(x\circ y-y\circ x)+3y\circ(x\star z)-3x\circ(y\star z)=0,
\end{equation}
for all $x,y,z\in A$.
\end{defi}

\begin{pro}
    Let $(A,\cdot,[-,-])$ be a transposed Poisson algebra, $(A,\star)$ be a compatible anti-Zinbiel algebra of $(A,\cdot)$ and $(A,\circ)$ be a compatible pre-Lie algebra of $(A,[-,-])$. If $(\mathcal{L}^{*}_{\star},\mathcal{L}^{*}_{\circ},A^{*})$ is a representation of $(A,\cdot,[-,-])$, then   $(A,\star,[-,-])$ is a TAPD algebra. Conversely, let $(A,\star,\circ)$ be a TAPD algebra, $(A,\cdot)$ be the sub-adjacent commutative associative algebra of $(A,\star)$ and $(A,[-,-])$ be the sub-adjacent Lie algebra of $(A,\circ)$. Then $(A,\cdot,[-,-])$ is a transposed Poisson algebra with a representation $(\mathcal{L}^{*}_{\star},\mathcal{L}^{*}_{\circ},A^{*})$. In this case, we say $(A,\cdot,[-,-])$ is the \textbf{sub-adjacent transposed Poisson algebra} of $(A,\star,\circ)$, and $(A,\star,\circ)$ is a \textbf{compatible TAPD algebra} of $(A,\cdot,[-,-])$.
\end{pro}
\begin{proof}
   Since $(\mathcal{L}^{*}_{\star},\mathcal{L}^{*}_{\circ},A^{*})$ is a representation of $(A,\cdot,[-,-])$, we get Eqs.~(\ref{TZPD2}) and (\ref{TAPD1}).
      By Eq.~(\ref{TAPD1}), Eq.~(\ref{TZPD4}) holds. Thus for all $x,y,z\in A$, we have
    \begin{eqnarray*}
        0&&=2z\cdot[x,y]-[z\cdot x,y]-[x,z\cdot y]\\
        &&=2[x,y]\star z+2z\star[x,y]-(z\cdot x)\circ y+y\circ(z\cdot x)-x\circ(z\cdot y)+(z\cdot y)\circ x\\
        &&\overset{(\ref{TZPD2})}{=}x\star(y\circ z)-y\star(x\circ z)+2z\star(x\circ y)-2z\star(y\circ x)\\
        &&\ \ \ \ -(z\star x)\circ y-(x\star z)\circ y+y\circ(z\star x)+y\circ(x\star z)\\
        &&\ \ \ \ -x\circ(z\star y)-x\circ(y\star z)+(y\star z)\circ x+(z\star y)\circ x\\
        &&\overset{(\ref{TAPD1}),(\ref{TZPD4})}{=}  x\circ(z\star y-3y\star z)-y\circ(z\star x-3x\star z)+z\star(x\circ y-y\circ x).
    \end{eqnarray*}
Hence Eq.~(\ref{TAPD3}) holds, and thus $(A,\star,\circ)$ is a TAPD algebra. The converse part is proved similarly.
\end{proof}

Hence  we get the following conclusion.

\begin{cor}
    Let $A$ be a vector space with two bilinear operations $\star,\circ:A\otimes A\rightarrow A$. Then the following conditions are equivalent:
    \begin{enumerate}
        \item $(A,\star,\circ)$ is a  TAPD algebra.
        \item The triple $(A,\cdot,[-,-])$   is a transposed Poisson algebra with a representation $(\mathcal{L}^{*}_{\star},\mathcal{L}^{*}_{\circ}$,
        $A^{*})$, where $\cdot$ and $[-,-]$ are respectively defined by Eqs.~(\ref{comm op}) and (\ref{sub-adj}).
        \item There is a transposed Poisson algebra structure on $A\oplus A^{*}$ in which the commutative associative product $\cdot$ is defined by  Eq.~(\ref{anti-Zinbiel axiom2}) and the Lie bracket $[-,-]$ is defined by Eq.~(\ref{pre-Lie axiom2}).
    \end{enumerate}
\end{cor}

\subsection{TAAD algebras}

\begin{defi}
    A \textbf{TAAD algebra} is a triple $(A,\star,\circ)$, such that $(A,\star)$ is an anti-Zinbiel algebra, $(A,\circ)$ is an anti-pre-Lie algebra, and
    Eqs.~(\ref{TZAD2}), (\ref{TAPD1}) and the following equation hold:
    \begin{equation}\label{TAAD3}
        x\circ(y\star z+z\star y)-y\circ(x\star z+z\star x)+z\star(x\circ y-y\circ x)=0,\;\;\forall x,y,z\in A.
    \end{equation}
\end{defi}

\begin{pro}
    Let $(A,\cdot,[-,-])$ be a transposed Poisson algebra, $(A,\star)$ be a compatible anti-Zinbiel algebra of $(A,\cdot)$ and $(A,\circ)$ be a compatible anti-pre-Lie algebra of $(A,[-,-])$. If $(\mathcal{L}^{*}_{\star},-\mathcal{L}^{*}_{\circ},A^{*})$ is a representation of $(A,\cdot,[-,-])$, then $(A,\star,[-,-])$ is a TAAD algebra. Conversely, let $(A,\star,\circ)$ be a TAAD algebra, $(A,\cdot)$ be the sub-adjacent commutative associative algebra of $(A,\star)$ and $(A,[-,-])$ be the sub-adjacent Lie algebra of $(A,\circ)$. Then $(A,\cdot,[-,-])$ is a transposed Poisson algebra with a representation $(\mathcal{L}^{*}_{\star},-\mathcal{L}^{*}_{\circ},A^{*})$. In this case, we say $(A,\cdot,[-,-])$ is the \textbf{sub-adjacent transposed Poisson algebra} of $(A,\star,\circ)$, and $(A,\star,\circ)$ is a \textbf{compatible  TAAD algebra} of $(A,\cdot,[-,-])$.
\end{pro}
\begin{proof}
    Since $(\mathcal{L}^{*}_{\star},-\mathcal{L}^{*}_{\circ},A^{*})$ is a representation of $(A,\cdot,[-,-])$, we get Eqs.~(\ref{TZAD2}) and (\ref{TAPD1}).
    Thus for all $x,y,z\in A$, we have
    \begin{eqnarray*}
        0&&=2z\cdot[x,y]-[z\cdot x,y]-[x,z\cdot y]\\
        &&=2[x,y]\star z+2z\star[x,y]-(z\cdot x)\circ y+y\circ(z\cdot x)-x\circ(z\cdot y)+(z\cdot y)\circ x\\
        &&\overset{(\ref{TZAD2})}{=}y\star(x\circ z)-x\star(y\circ z)+2z\star(x\circ y)-2z\star(y\circ x)\\
        &&\ \ \ \ -(z\star x)\circ y-(x\star z)\circ y+y\circ(z\star x)+y\circ(x\star z)\\
        &&\ \ \ \ -x\circ(z\star y)-x\circ(y\star z)+(y\star z)\circ x+(z\star y)\circ x\\
        &&\overset{(\ref{TAPD1}),(\ref{TZPD4})}{=}x\circ(y\star z+z\star y)-y\circ(x\star z+z\star x)+z\star(x\circ y-y\circ x).
    \end{eqnarray*}
Hence Eq.~(\ref{TAAD3}) holds, and thus $(A,\star,\circ)$ is a TAAD algebra. The converse part is proved similarly.
\end{proof}

Hence  we get the following conclusion.

\begin{cor}
    Let $A$ be a vector space with two bilinear operations $\star,\circ:A\otimes A\rightarrow A$. Then the following conditions are equivalent:
    \begin{enumerate}
        \item $(A,\star,\circ)$ is a   TAAD algebra.
        \item The triple $(A,\cdot,[-,-])$   is a transposed Poisson algebra with a representation $(\mathcal{L}^{*}_{\star},-\mathcal{L}^{*}_{\circ}$,
        $A^{*})$, where $\cdot$ and $[-,-]$ are respectively defined by Eqs.~(\ref{comm op}) and (\ref{sub-adj}).
        \item There is a transposed Poisson algebra structure on $A\oplus A^{*}$ in which the commutative associative product $\cdot$ is defined  by Eq.~(\ref{anti-Zinbiel axiom2}) and the Lie bracket $[-,-]$ is defined by Eq.~(\ref{anti-pre-Lie axiom2}).
    \end{enumerate}
\end{cor}

\begin{pro}
    Let $(A,\cdot,[-,-])$ be a transposed Poisson algebra. Suppose that $\mathcal{B}$ is a nondegenerate symmetric bilinear form on $A$ such that it is a commutative Connes cocycle on $(A,\cdot)$ and a commutative 2-cocycle on $(A,[-,-])$. Then there is a compatible TAAD algebra $(A,\star,\circ)$ in which $\star$ and $\circ$ are respectively defined   by Eqs.~(\ref{bilinear form anti-Zinbiel}) and (\ref{bilinear form anti-pre-Lie}) . Conversely, let $(A,\star,\circ)$ be a   TAAD algebra and the sub-adjacent transposed Poisson algebra be $(A,\cdot,[-,-])$. Then there is a transposed Poisson algebra $A\ltimes_{\mathcal{L}^{*}_{\star},-\mathcal{L}^{*}_{\circ}}A^{*}$, and the natural nondegenerate symmetric bilinear form $\mathcal{B}_{d}$ defined by Eq.~(\ref{B_{d}})
    is a commutative Connes cocycle on the commutative associative algebra $A\ltimes_{\mathcal{L}^{*}_{\star}}A^{*}$ and a commutative 2-cocycle on the Lie algebra $A\ltimes_{-\mathcal{L}^{*}_{\circ}}A^{*}$.
\end{pro}
\begin{proof}
It is similar to the proof of Proposition \ref{pro:anti-pre-Lie Poisson} given in \cite{LB2022}.
\end{proof}


\subsection{Summary}\

We summarize some facts on the 8 algebraic structures in the previous subsections respectively corresponding to the mixed splittings of operations of transposed Poisson algebras in terms of the representations of transposed Poisson algebras on the dual spaces in Table 3.

\begin{small}
    \begin{table}[!t]
        \caption{Splittings of transposed Poisson algebras on dual spaces}
        %
        %
\begin{tabular}{p{1.5cm}p{2cm}p{3.5cm}p{7cm}}
   \hline
    Algebras & Notations & Representations of

    transposed Poisson

    algebras  on

    the dual spaces & {Corresponding nondegenerate bilinear

    forms on transposed Poisson algebras} \\
    TCPD   & $(A,\cdot,\circ)$  & $(-\mathcal{L}^{*}_{\cdot},\mathcal{L}^{*}_{\circ},A^{*})$ & - \\

    TCAD  & $(A,\cdot,\circ)$ & $(-\mathcal{L}^{*}_{\cdot},-\mathcal{L}^{*}_{\circ},A^{*})$ & invariant,

    commutative 2-cocycle \\

    TZLD & $(A,\star,[-,-])$ & $(-\mathcal{L}^{*}_{\star},\mathrm{ad}^{*},A^{*})$ & -\\

    TZPD & $(A,\star,\circ)$ & $(-\mathcal{L}^{*}_{\star},\mathcal{L}^{*}_{\circ},A^{*})$ & Connes cocycle,

    symplectic form\\

    TZAD & $(A,\star,\circ)$ & $(-\mathcal{L}^{*}_{\star},-\mathcal{L}^{*}_{\circ},A^{*})$ & -\\

    TALD & $(A,\star,[-,-])$ & $(\mathcal{L}^{*}_{\star},\mathrm{ad}^{*},A^{*})$ & commutative Connes cocycle,

    invariant\\

    TAPD & $(A,\star,\circ)$ & $(\mathcal{L}^{*}_{\star},\mathcal{L}^{*}_{\circ},A^{*})$ & -\\

    TAAD & $(A,\star,\circ)$ & $(\mathcal{L}^{*}_{\star},-\mathcal{L}^{*}_{\circ},A^{*})$ & commutative Connes cocycle,

    commutative 2-cocycle\\
   \hline
\end{tabular}
\end{table}
\end{small}

\bigskip

{\bf Acknowledgments } This work is partially supported by NSFC
(11931009, 12271265, 12261131498), the Fundamental Research Funds for the
Central Universities and Nankai Zhide Foundation. The authors thank Professor Li Guo for
valuable discussion.

\end{document}